\documentclass[11pt]{article}

\usepackage[letterpaper,margin=1in]{geometry}
\usepackage[pdfpagelabels,bookmarks=true]{hyperref}

\usepackage{amsmath, amssymb}
\usepackage{amsthm}

\hypersetup{colorlinks, linkcolor=darkblue, citecolor=darkgreen, urlcolor=darkblue}
\newcommand{\footremember}[2]{\footnote{#2}
    \newcounter{#1}
    \setcounter{#1}{\value{footnote}}}
\newcommand{\footrecall}[1]{\footnotemark[\value{#1}]}

\usepackage{anyfontsize}
\usepackage{complexity}

\usepackage{xspace}

\usepackage[inline]{enumitem}
\setlist[enumerate,1]{label=(\roman*), leftmargin=2.2em}
\setlist[enumerate,2]{label=(\alph*)}
\setlist{nosep,topsep=0.1em}
\setlist[itemize,1]{label={\bfseries--}}

\usepackage{mathtools}

\makeatletter
\newtheorem*{rep@theorem}{\rep@title}
\newcommand{\newreptheorem}[2]{\newenvironment{rep#1}[1]{\def\rep@title{\cref{##1}}\begin{rep@theorem}}{\end{rep@theorem}}}
\makeatother

\newreptheorem{theorem}{Theorem}
\newreptheorem{lemma}{Lemma}
\newreptheorem{corollary}{Corollary}

\usepackage{mdframed}

\usepackage{xcolor}
\definecolor{darkblue}{rgb}{0,0,0.38}
\definecolor{darkred}{rgb}{0.6,0,0}
\definecolor{darkgreen}{rgb}{0.1,0.35,0}

\usepackage[
backend=biber,
style=alphabetic,
citestyle=alphabetic,
maxalphanames=6,
maxcitenames=99,
mincitenames=98,
maxbibnames=99,
giveninits=true,
url=false,
doi=true,
isbn=true,
backref=true
]{biblatex}

\usepackage{xpatch}

\makeatletter
\patchcmd\blx@bblinput{\blx@blxinit}
                      {\blx@blxinit
                      }{}{\fail}
\makeatother

\makeatletter
\DeclareFieldFormat{eprint:arxiv}{arXiv\addcolon\space
  \ifhyperref
    {\href{https://arxiv.org/abs/#1}{\nolinkurl{#1}\iffieldundef{eprintclass}
     {}
{\addspace\mkbibbrackets{\thefield{eprintclass}}}}}
    {\nolinkurl{#1}
     \iffieldundef{eprintclass}
       {}
{\addspace\mkbibbrackets{\thefield{eprintclass}}}}}
\makeatother

\usepackage[vlined,ruled,algo2e,linesnumbered]{algorithm2e}
\SetAlCapHSkip{0.2em}
\SetAlgoInsideSkip{smallskip}
\SetCustomAlgoRuledWidth{0.95\textwidth}
\makeatletter
\patchcmd{\@algocf@start}{\begin{lrbox}{\algocf@algobox}}{\rule{0.025\textwidth}{\z@}\begin{lrbox}{\algocf@algobox}\begin{minipage}{0.95\textwidth}}{}{}
\patchcmd{\@algocf@finish}{\end{lrbox}}{\end{minipage}\end{lrbox}}{}{}
\makeatother

\usepackage[margin=10pt,font=small,labelfont=bf]{caption}
\usepackage[margin=10pt,font=small,labelfont=bf]{subcaption}

\usepackage{graphicx}
\graphicspath{{.}{graphics/}}
\makeatletter
\newcommand\appendtographicspath[1]{\g@addto@macro\Ginput@path{#1}}
\makeatother

\usepackage{tikz}
\usepackage{pgfplots}
\usetikzlibrary{calc}
\usetikzlibrary{positioning}
\usetikzlibrary{decorations.pathmorphing}
\usetikzlibrary{decorations.pathreplacing}
\usetikzlibrary{decorations.markings}
\usetikzlibrary{backgrounds}
\usetikzlibrary{shapes.misc, shapes.geometric}

\makeatletter
\DeclareRobustCommand{\cev}[1]{{\mathpalette\do@cev{#1}}}
\newcommand{\do@cev}[2]{\vbox{\offinterlineskip
    \sbox\z@{$\m@th#1 x$}\ialign{##\cr
      \hidewidth\reflectbox{$\m@th#1\vec{}\mkern4mu$}\hidewidth\cr
      \noalign{\kern-\ht\z@}
      $\m@th#1#2$\cr
    }}}
\makeatother

\usepackage[capitalize, nameinlink]{cleveref}

\newtheorem*{theorem*}{Theorem}
\newtheorem{theorem}{Theorem}
\newtheorem{lemma}[theorem]{Lemma}
\newtheorem*{lemma*}{Lemma}
\newtheorem{conjecture}[theorem]{Conjecture}

\newtheorem{definition}[theorem]{Definition}
\newtheorem{remark}[theorem]{Remark}
\newtheorem{corollary}[theorem]{Corollary}

\crefname{theorem}{Theorem}{Theorems}
\crefname{conjecture}{Conjecture}{Conjectures}
\Crefname{lemma}{Lemma}{Lemmas}
\Crefname{claim}{Claim}{Claims}
\Crefname{fact}{Fact}{Facts}
\Crefname{remark}{Remark}{Remarks}
\Crefname{observation}{Observation}{Observations}
\Crefname{line}{Line}{Lines}
\Crefname{figure}{Figure}{Figures}
\crefalias{AlgoLine}{line}

\newcommand{\horizontal}{\mathcal{H}}
\usepackage{multicol} \usetikzlibrary{calc}

\usepackage{etoolbox}

\newbool{towerlabels}
\setbool{towerlabels}{true}

\tikzset{matching/.style={red, line width=4pt}}

\newcommand{\convexpath}[2]{
[
    create hullnodes/.code={
        \global\edef\namelist{#1}
        \foreach [count=\counter] \nodename in \namelist {
            \global\edef\numberofnodes{\counter}
            \node at (\nodename) [draw=none,name=hullnode\counter] {};
        }
        \node at (hullnode\numberofnodes) [name=hullnode0,draw=none] {};
        \pgfmathtruncatemacro\lastnumber{\numberofnodes+1}
        \node at (hullnode1) [name=hullnode\lastnumber,draw=none] {};
    },
    create hullnodes
]
($(hullnode1)!#2!-90:(hullnode0)$)
\foreach [
    evaluate=\currentnode as \previousnode using \currentnode-1,
    evaluate=\currentnode as \nextnode using \currentnode+1
    ] \currentnode in {1,...,\numberofnodes} {
-- ($(hullnode\currentnode)!#2!-90:(hullnode\previousnode)$)
  let \p1 = ($(hullnode\currentnode)!#2!-90:(hullnode\previousnode) - (hullnode\currentnode)$),
    \n1 = {atan2(\y1,\x1)},
    \p2 = ($(hullnode\currentnode)!#2!90:(hullnode\nextnode) - (hullnode\currentnode)$),
    \n2 = {atan2(\y2,\x2)},
    \n{delta} = {-Mod(\n1-\n2,360)}
  in
    {arc [start angle=\n1, delta angle=\n{delta}, radius=#2]}
}
-- cycle
}

\newcommand{\flippath}[2]{
	[
	create pathvertexs/.code={
		\global\edef\namelistOuter{#1}
		\foreach [count=\counterOuter] \vertexname in \namelistOuter {
			\global\edef\numberofvertexs{\counterOuter};
			\coordinate(pathvertex\counterOuter) at (\vertexname);
		}
		\coordinate(pathvertex0) at (pathvertex\numberofvertexs);
	},
	create pathvertexs
	]
	\pgfmathtruncatemacro\endloop{\numberofvertexs};
	\foreach [
	evaluate=\currentvertex as \previousvertex using \currentvertex-1
	] \currentvertex in {2,...,\endloop} {
\fill \convexpath{pathvertex\currentvertex,pathvertex\previousvertex}{#2};
	}
}

\newcommand{\tower}[7]{

        \begin{scope}[every node/.style={thick,draw=black,fill=white,circle,minimum size=8, inner sep=1pt}]
            \ifbool{towerlabels}{
            \foreach \i in {0,...,#3} {
                \coordinate (#7a\i) at (#1 + \i*#5 , #2+\i*#6);
                \node at (#7a\i) {$a_{\i}^{#4}$};
                \coordinate (#7b\i) at (#1 + #6 + \i*#5, #2 - #5 +\i*#6);
                \node at (#7b\i) {$b_{\i}^{#4}$};

            }
        }
        {
            \foreach \i in {0,...,#3} {

				\coordinate (#7a\i) at (#1 + \i*#5 , #2+\i*#6);
				\node at (#7a\i) {};
				\coordinate (#7b\i) at (#1 + #6 + \i*#5, #2 - #5 +\i*#6);
				\node at (#7b\i) {};

            }
        }
        \end{scope}

		\begin{pgfonlayer}{background}
        \begin{scope}[line width=2pt]

		\ifthenelse{#3>0}{
            \foreach \i in {1,...,#3} {
                \pgfmathtruncatemacro{\previ}{\i - 1}
                \draw (#7a\i) -- (#7b\i);
                \draw (#7a\i) -- (#7a\previ);
                \draw (#7b\i) -- (#7b\previ);

            }}{}

            \draw (#7a0) -- (#7b0);
        \end{scope}
    \end{pgfonlayer}

        \stepcounter{towerIndex}

}

\newcommand{\towerRatio}{0.5}
\newcommand{\edgeWithTowers}[6]{

    \pgfmathsetmacro{\xdirection}{#3 - #1}
    \pgfmathsetmacro{\ydirection}{#4 - #2}

    \pgfmathsetmacro{\length}{veclen(\xdirection,\ydirection)}
    \pgfmathsetmacro{\normalxdirection}{\xdirection / ((#5))}
    \pgfmathsetmacro{\normalydirection}{\ydirection / ((#5))}
    \pgfmathsetmacro{\towerxdirection}{\towerRatio*\normalxdirection}
    \pgfmathsetmacro{\towerydirection}{\towerRatio*\normalydirection}

    \pgfmathsetmacro{\xoffset}{(\xdirection - #5*\towerxdirection)/ (#5+1)}
    \pgfmathsetmacro{\yoffset}{(\ydirection - #5*\towerydirection)/ (#5+1)}

    \foreach \j in {1,...,#5} {
        \pgfmathsetmacro{\factor}{\j-1}

        \pgfmathsetmacro{\xPos}{#1+\factor*\towerxdirection + \j*\xoffset}
        \pgfmathsetmacro{\yPos}{#2+\factor*\towerydirection + \j*\yoffset}
        \tower{\xPos}{\yPos}{#6}{T_{\thetowerIndex}}{-\towerydirection}{\towerxdirection}{\thetowerIndex}
    }
}

\newcommand{\edgeWithTowersDifferentHeight}[6]{

\pgfmathsetmacro{\xdirection}{#3 - #1}
\pgfmathsetmacro{\ydirection}{#4 - #2}

\pgfmathsetmacro{\length}{veclen(\xdirection,\ydirection)}
\pgfmathsetmacro{\normalxdirection}{\xdirection / ((#5))}
\pgfmathsetmacro{\normalydirection}{\ydirection / ((#5))}
\pgfmathsetmacro{\towerxdirection}{\towerRatio*\normalxdirection}
\pgfmathsetmacro{\towerydirection}{\towerRatio*\normalydirection}

\pgfmathsetmacro{\xoffset}{(\xdirection - #5*\towerxdirection)/ (#5+1)}
\pgfmathsetmacro{\yoffset}{(\ydirection - #5*\towerydirection)/ (#5+1)}

\global\edef\namelist{#6}
\foreach [count=\counter] \height in \namelist {
	\pgfmathsetmacro{\factor}{\counter-1}
	
	\pgfmathsetmacro{\xPos}{#1+\factor*\towerxdirection + \counter*\xoffset}
	\pgfmathsetmacro{\yPos}{#2+\factor*\towerydirection + \counter*\yoffset}
	\tower{\xPos}{\yPos}{\height}{T_{\thetowerIndex}}{-\towerydirection}{\towerxdirection}{\thetowerIndex}
}
} 
\title{\Large Complexity of polytope diameters via perfect matchings}

\author{
Christian N{\"o}bel\footremember{ETH}{
ETH Zurich, Zurich, Switzerland. R.S. funded by SNSF Ambizione Grant No. 216071, C.N. funded by the European Research Council (ERC) under the European Union's Horizon 2020 research and innovation programme (grant agreement No 817750).
Email: $\{$\href{mailto:cnoebel@ethz.ch}{cnoebel}, \href{mailto:rsteine@ethz.ch}{rsteine}$\}$@ethz.ch.}\and
Raphael Steiner\footrecall{ETH}}

\date{}

\begin{document}

\newcounter{towerIndex}
\renewcommand{\tilde}{\widetilde}

\maketitle

\begin{abstract}
 The \emph{Circuit diameter} of polytopes was introduced by Borgwardt, Finhold and Hemmecke \cite{BorgwardtFinholdHemmecke15} as a fundamental tool for the study of circuit augmentation schemes for linear programming and for estimating combinatorial diameters. Determining the complexity of computing the circuit diameter of polytopes was posed as an open problem by Sanità \cite{SanitaTalk2020} as well as by Kafer~\cite{Kafer22}, and was recently reiterated by Borgwardt, Grewe, Kafer, Lee and Sanità \cite{borgwardt2024hardness}.
 In this paper, we solve this problem by showing that computing the circuit diameter of a polytope given in halfspace-description is strongly \NP-hard. To prove this result, we show that computing the combinatorial diameter of the \emph{perfect matching polytope} of a bipartite graph is \NP-hard. This complements a result by Sanità (FOCS 2018,~\cite{Sanita18}) on the \NP-hardness of computing the diameter of fractional matching polytopes and implies the new result that computing the diameter of a $\{0,1\}$-polytope is strongly \NP-hard, which may be of independent interest. In our second main result, we give a precise graph-theoretic description of the \emph{monotone diameter} of perfect matching polytopes and use this description to prove that computing the monotone (circuit) diameter of a given input polytope is strongly \NP-hard as well.
\end{abstract}

\begin{tikzpicture}[overlay, remember picture, shift = {(current page.south east)}]
\coordinate (anchor) at (-2,0);
\node[anchor=south east, outer sep=5mm] at (anchor) {
    \begin{tikzpicture}[outer sep=0] \node (ERC) {\includegraphics[height=13mm]{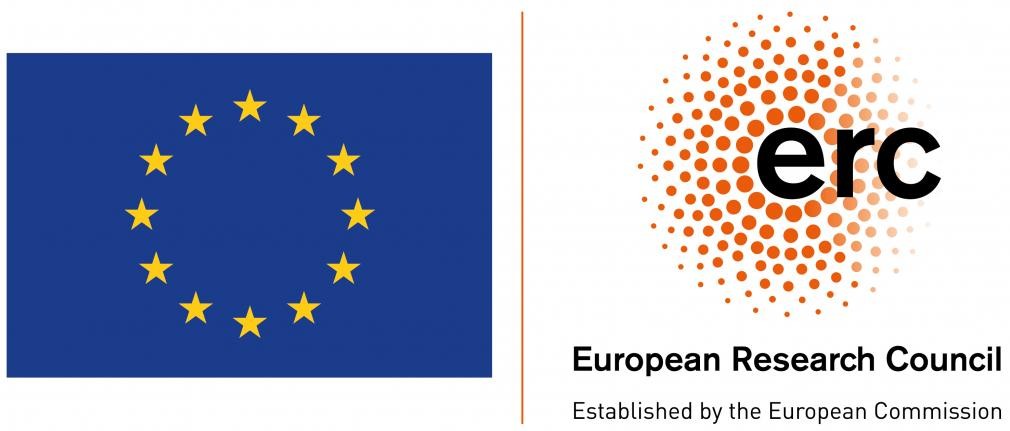}};
    \end{tikzpicture}
};
\end{tikzpicture}

\section{Introduction}

One of the most central open problems in the theory of mathematical optimization is \emph{Smale's 9th problem}, included in his list of open problems for the 21st century~\cite{Smale1998}. It asks for a strongly polynomial time algorithm for the linear programming problem, i.e., the algorithmic problem of optimizing a linear functional subject to linear inequality and equality constraints.

One of the canonical candidates that still holds potential for a positive resolution of Smale's problem is the famous \emph{simplex method}, whose basic form was invented by George Dantzig around 1950. Roughly speaking, to solve a given linear program, the simplex method updates an extreme point of the (polyhedral) feasible region while moving along its edges in such a way that the objective value is successively improved. In particular, the sequence of extreme points visited by the simplex method describes a path on the $1$-skeleton of the constraint-polyhedron. While efficient in practice, the theoretical complexity of the simplex method is a famous unsolved problem. In particular, the existence of a pivot rule that makes the simplex method run in strongly polynomial time remains open.

Suppose we would like to minimize a linear functional $\mathbf{c}^T\mathbf{x}$ over a polytope $P \subseteq \mathbb{R}^d$. Then the number of steps taken by an execution of the simplex algorithm using \emph{any} pivot-rule is lower-bounded by the minimum length of a path in the $1$-skeleton of $P$ connecting the starting vertex to the optimal solution. Since, by varying $\mathbf{c}$, every vertex of $P$ can be made the (unique) optimal solution of the corresponding linear program, and since the starting vertex in the simplex method can also be chosen arbitrarily, this shows that a lower bound for the complexity of the simplex method is given by the \emph{diameter} $\mathrm{diam}(P)$ of $P$, which is defined to be the diameter of the graph formed by the vertices and edges of $P$. In fact, since the simplex algorithm always follows a \emph{$\mathbf{c}$-monotone} path on the polytope, i.e., a path along which the objective value $\mathbf{c}^T\mathbf{x}$ is non-increasing, a stronger lower bound holds, which is called the \emph{monotone diameter} $\mathrm{mdiam}(P)$ of the polytope $P$.
To be precise, the monotone diameter of $P$ with respect to a cost vector $c$, $\mathrm{mdiam}(P,\mathbf{c})$, is defined to be the largest distance from any vertex of $P$ to a $\mathbf{c}$-minimal vertex, when only considering edge moves that do not strictly increase $\mathbf{c}^T\mathbf{x}$.
Then the monotone diameter of $P$, $\mathrm{mdiam}(P)$, is defined to be the maximum of $\mathrm{mdiam}(P,\mathbf{c})$ over all choices of $\mathbf{c}$.

As a consequence of the above discussion, a necessary condition for the existence of a strongly polynomial pivot rule is that the (monotone) diameter of every polytope $P\subseteq \mathbb{R}^d$ with $n$ facets is bounded by a polynomial function in $n$ and $d$. However, even this relaxed problem, known as the \emph{polynomial Hirsch conjecture}, remains unsolved. The classical \emph{Hirsch conjecture} stated that every $d$-dimensional polytope $P$ with $n$ facets satisfies $\mathrm{diam}(P)\le n-d$, but in 2012 Santos constructed a family of counterexamples to this conjecture~\cite{santos}.

\paragraph{Complexity of the diameter of polytopes.}

Our results address a long-standing open question in linear programming and discrete geometry, namely to determine the complexity of computing the (monotone) diameter of a given input polytope (as references, see e.g. Problem~10 in the survey article by Kaibel and Pfetsch~\cite{KaibelPfetsch2003} as well as the list of open problems on this problem posed by Frieze and Teng~\cite{FriezeTeng1994}).
More precisely, we will consider the following computational problems.

\begin{mdframed}[innerleftmargin=0.5em, innertopmargin=0.5em, innerrightmargin=0.5em, innerbottommargin=0.5em, userdefinedwidth=0.95\linewidth, align=center]
	{\textbf{DIAMETER}}
	\sloppy

	\textbf{Input:} A matrix $A\in\mathbb{Q}^{m\times d}$ and a vector $\mathbf{b}\in\mathbb{Q}^m$, defining the polytope $$P = \{\mathbf{x}\in \mathbb{R}^d|A\mathbf{x} \leq \mathbf{b}\}.$$

	\textbf{Output:} $\mathrm{diam}(P)$.
\end{mdframed}

\begin{mdframed}[innerleftmargin=0.5em, innertopmargin=0.5em, innerrightmargin=0.5em, innerbottommargin=0.5em, userdefinedwidth=0.95\linewidth, align=center]
	{\textbf{MONOTONE DIAMETER}}
	\sloppy

	\textbf{Input:} A matrix $A\in\mathbb{Q}^{m\times d}$ and a vector $\mathbf{b}\in\mathbb{Q}^m$, defining the polytope $$P = \{\mathbf{x}\in \mathbb{R}^d|A\mathbf{x} \leq \mathbf{b}\}.$$

	\textbf{Output:} $\mathrm{mdiam}(P)$.
\end{mdframed}

The classical result in this direction is due to Frieze and Teng~\cite{FriezeTeng1994} from 1994, who showed that \textsc{Diameter} is \emph{weakly} \NP-hard. In a more recent breakthrough, Sanità~\cite{Sanita18} strengthened this result by showing that the same problem is in fact \emph{strongly} \NP-hard.
More precisely, to prove this result, Sanità analyzed the diameter of so-called \emph{fractional matching polytopes} of graphs and proved that computing the diameter for this special class of polytopes is \NP-hard.

The first main contribution of this paper is a novel short proof of the strong \NP-hardness of \textsc{Diameter}.
This will be done by showing that computing the diameter of the perfect matching polytope of a bipartite graph is strongly \NP-hard.
The latter form a well-known class of combinatorial polytopes that have been widely studied in combinatorial optimization, in particular due to their strong connection to the maximum weight perfect matching problem and the network simplex method.
The following gives a formal definition.

\begin{definition}
	Let $G=(V,E)$ be a bipartite graph. For every perfect matching $M$ of $G$, define the vector $\chi^M\in \mathbb{R}^E$ that has a $1$-entry for every $e \in M$ and all other entries equal to $0$. The \emph{perfect matching polytope} $P_G$ associated with $G$ is the polytope in $\mathbb{R}^E$ defined as the convex hull
	$$P_G=\mathrm{conv}\{\chi^M|M\text{ perfect matching of }G\}.$$
\end{definition}

It is well-known (see, e.g. Chapter~18 in~\cite{schrijver2003}) that for every bipartite graph $G=(V,E)$, the polytope $P_G$ also admits a compact halfspace-encoding. Namely, an edge-indexed vector $(x_e)_{e \in E} \in \mathbb{R}^E$ belongs to $P_G$ if and only if the following hold.

\begin{align*}
	\sum_{e \ni v}{x_e} &= 1,  \qquad (\forall v \in V) \\
	x_e &\ge 0, \qquad (\forall e \in E).
\end{align*}

The following is our first main result.

\begin{theorem}\label{thm:perfectmatchingpolytope}
	The following problem is \NP-hard: Given as input a bipartite graph $G$, determine the diameter of the associated perfect matching polytope $P_G$.
\end{theorem}

To the best of our knowledge the complexity of \textsc{Monotone Diameter} has remained unknown thus far.
In our second main result we give a combinatorial description of the monotone diameter of perfect matching polytopes of bipartite graphs in terms of a special cycle packing invariant.
This description in particular generalizes a result of Rispoli~\cite{Rispoli1992} on the monotone diameter of the so-called assignment polytope $P_{K_{n,n}}$.
Using this description we then show that computing the monotone diameter of perfect matching polytopes is \NP-hard as well.

\begin{theorem}\label{thm:monotoneperfectmatchingpolytope}
	The following problem is \NP-hard: Given as input a bipartite graph $G$, determine the monotone diameter of the associated perfect matching polytope $P_G$.
\end{theorem}

In contrast to the fractional matching polytopes studied in \cite{Sanita18}, perfect matching polytopes of bipartite graphs are 0/1-polytopes defined by a totally unimodular constraint matrix.
This in particular has the following consequence.

\begin{corollary}
	Determining the diameter and the monotone diameter of 0/1-polytopes are both strongly \NP-hard problems, even when restricted to polytopes with a totally unimodular constraint matrix.
\end{corollary}

\paragraph{Applications to the circuit diameter.}
Our results have further implications for the complexity of a well-studied variant of the combinatorial diameter, known as the circuit-diameter of polytopes.
This concept was introduced by Borgwardt, Finhold and Hemmecke~\cite{BorgwardtFinholdHemmecke15}, motivated by the recently popular concepts of \emph{circuit moves} and \emph{circuit augmention schemes} which, as the simplex method, provide a framework for designing algorithms for linear programming. Roughly speaking, circuit moves extend the simplex paradigm of traversing the edges of a polytope, by allowing to follow additional directions through the polytope, called \emph{circuits}.
The \emph{circuit diameter} of a polytope $P$, denoted $\mathrm{cdiam}(P)$, is the maximum  distance between any pair of vertices of $P$ when moving along circuits.
Several pivot-rules and algorithms for linear and combinatorial optimization based on circuit moves have been proposed and studied recently, see e.g.~\cite{deLoera2015,BORGWARDT2020decent,BORGWARDT2021DECENT,DeLoeraKaferSanita22,Dadush22,borgwardt2023combinatorial}.
For the sake of completeness we briefly recall the formal definition of circuit moves and circuit diameter.

\begin{definition}[cf. Definitions~1--5 in~\cite{DeLoeraKaferSanita22}]\label{def:circuitmove}
	Let $P=\{\mathbf{x}\in \mathbb{R}^d|A\mathbf{x}=\mathbf{b}, B\mathbf{x} \le \mathbf{d}\}$ be a polyhedron.

	\begin{enumerate}
		\item A \emph{circuit} of $P$ is a vector $\mathbf{g} \in \mathbb{R}^d\setminus \{\mathbf{0}\}$ such that
		\begin{itemize}
			\item $A\mathbf{g}=0$, and
			\item $B\mathbf{g}$ is inclusion-wise support-minimal in the collection $\{B\mathbf{y}|A\mathbf{y}=\mathbf{0}, \mathbf{y} \neq \mathbf{0}\}$.
		\end{itemize}
		\item Given a point $\mathbf{x} \in P$, a \emph{circuit move} at $\mathbf{x}$ consists of selecting a circuit $\mathbf{g}$ of $P$ and moving to a new point $\mathbf{x}'=\mathbf{x}+\alpha \mathbf{g}$, where $\alpha > 0$ is \emph{maximal} w.r.t.\ $\mathbf{x}+\alpha \mathbf{g} \in P$.
		\item A \emph{circuit walk} of length $k$ is a sequence $\mathbf{x}_0,\mathbf{x}_1,\ldots,\mathbf{x}_k$ of points in $P$ such that for every $i=1,\ldots,k$, we have that $\mathbf{x}_i$ is obtained from $\mathbf{x}_{i-1}$ by a circuit move.
		\item Given two points $\mathbf{x},\mathbf{x}'\in P$, the \emph{circuit distance} $\mathrm{cdist}(\mathbf{x},\mathbf{x}')$ is defined as the minimum length of a circuit walk that starts in $\mathbf{x}$ and ends in $\mathbf{x}'$.

		\item The \emph{circuit diameter} of a polytope $P$, denoted $\mathrm{cdiam}(P)$, is the maximum circuit distance among all pairs of vertices of $P$.
	\end{enumerate}
\end{definition}

In the same way as the ordinary diameter of a polytope $P$ forms a lower bound on the run-time of the simplex-method for linear programs with feasible region $P$, the circuit diameter $\mathrm{cdiam}(P)$ forms a lower bound on the time complexity of circuit augmentation schemes for linear optimization over $P$. This motivates bounding the circuit diameter of $d$-dimensional polytopes with $n$ facets by a polynomial function in $n$ and $d$. While this remains open, DeLoera, Kafer and Sanità~\cite{DeLoeraKaferSanita22} could recently prove the remarkable result that the circuit diameter of a polytope is polynomially bounded in terms of $n$, $d$ and the maximum encoding-length among the coefficients in its description.  Interestingly, the analogue of the classical Hirsch conjecture for circuit diameters, proposed in 2015 by Borgwardt, Finhold and Hemmecke~\cite{BorgwardtFinholdHemmecke15}, has still neither been proved nor disproved.
\begin{conjecture}[\cite{BorgwardtFinholdHemmecke15}]\label{con:circuitdiameter}
	Every polytope $P\subseteq \mathbb{R}^d$ with $n$ facets satisfies $\mathrm{cdiam}(P)\le n-d$.
\end{conjecture}
Many more results on the circuit diameter, including general bounds as well as for special classes of polytopes, have been obtained recently, see~\cite{BorgwardtFinholdHemmecke15,StephenYusun15,BorgwardtStephenYusun18,KaferPashkovichSanita19,Dadush22,black2023circuit} for only a few examples.

The computational complexity of computing the Circuit Diameter, and in particular whether computing the {Circuit Diameter} is \NP-hard, was raised as an open problem by Sanità~\cite{SanitaTalk2020} as well as by Kafer~\cite{Kafer22}. Very recently, the problem was reiterated by Borgwardt et al.~\cite{borgwardt2024hardness}.
This problem was a further motivation for our study of perfect matching polytopes.  Namely, as already observed in~\cite{BORGWARDT2022,cardinal_steiner_23}, the notions of (monotone)\footnote{The monotone circuit diameter is defined analogously to the monotone diameter, when replacing edge moves by circuit moves} circuit diameter and (monotone) combinatorial diameter agree on these polytopes.
\begin{theorem*}[cf. Corollary~4 in~\cite{BORGWARDT2022}, Lemma~2 in~\cite{cardinal_steiner_23}]\label{thm:circuitmoves}
	Let $G$ be a bipartite graph, let $\mathbf{x}$ be a vertex of $P_G$, and let $\mathbf{x}'\in P_G$. Then $\mathbf{x}'$ can be obtained from $\mathbf{x}$ by a circuit move if and only if $\mathbf{x}'$ is also a vertex of $P_G$ and adjacent to $\mathbf{x}$ on the skeleton of $P_G$.
\end{theorem*}
In particular, combining this with \cref{thm:perfectmatchingpolytope} and \cref{thm:monotoneperfectmatchingpolytope}, we can immediately answer the aforementioned question on the \NP-hardness of the Circuit Diameter affirmatively:

\begin{theorem}\label{thm:main}
	Computing the {(monotone) Circuit Diameter} of a polytope is strongly \NP-hard, even when restricted to the class of perfect matching polytopes of bipartite graphs.
\end{theorem}

\paragraph{Related work.} Cardinal and Steiner~\cite{cardinal_steiner_23}  proved that (monotone) circuit distances in perfect matching polytopes cannot be approximated to within any constant factor, unless $\P=\NP$. In a similar direction, Borgwardt et al.~\cite{borgwardt2024hardness} recently proved that computing exact circuit distances is NP-hard for $0/1$-network flow polytopes.  While conceptually related to our \cref{thm:perfectmatchingpolytope,thm:monotoneperfectmatchingpolytope}, the results are incomparable. In particular, one major difference between the approaches in~\cite{cardinal_steiner_23,borgwardt2024hardness} and the diameter problem considered here is that the reductions in~\cite{cardinal_steiner_23,borgwardt2024hardness} are based on \emph{very} short paths, namely of length two, while our reduction here naturally needs to deal with \emph{maximal distances} and thus with rather long paths in the perfect matching polytope. Strengthening \cref{thm:perfectmatchingpolytope,thm:monotoneperfectmatchingpolytope} to an inapproximability result for the diameter would be quite interesting, but this seems to be outside the reach of the methods used in our proof.

\section{Overview of the reductions}

In this section, we present the reductions used to obtain our main results.
The technical proofs will be postponed to later sections.

\subsection{Reduction for \textsc{Diameter}}\label{sec:reduction-diameter}
We start by stating the simple characterization of adjacency on perfect matching polytopes of bipartite graphs, which reduces the analysis of the diameter of these polytopes to a purely graph-theoretic issue.

\begin{lemma*}[cf.~\cite{CHVATAL1975}, \cite{IWATA2002}]
	Let $G=(V,E)$ be a bipartite graph.
	Consider two vertices $\mathbf{x}=\chi^M$ and $\mathbf{y}=\chi^{M'}$ of $P_G$ corresponding to perfect matchings $M$ and $M'$ in $G$.
	Then $\mathbf{x}$ and $\mathbf{y}$ are adjacent in the skeleton of $P_G$ if and only if the symmetric difference $M\Delta M'$ consists of the edge-set of exactly one cycle in $G$.
\end{lemma*}

From the above we can see that the problem of determining the (combinatorial) diameter of the perfect matching polytope of $G$ boils down to determining the maximal distance of two matchings in $G$, when one is allowed to flip\footnote{If $C$ is a cycle in $G$ that is alternating w.r.t.~the current perfect matching $M$, then flipping $C$ means exchanging matching and non-matching edges along $C$, i.e., moving from $M$ to the new perfect matching given by the symmetric difference $M\Delta C$.} a single alternating cycle at a time.

We will reduce the \textsc{Hamiltonian Cycle Problem} to the problem of determining the diameter of the perfect matching polytope $P_G$ of a given graph $G$. Recall that the Hamiltonian cycle problem is the algorithmic decision problem of deciding whether or not an input graph contains a Hamiltonian cycle, and that this problem is well-known to be \NP-complete~\cite{GareyJohnson1990}.

Let an input graph $H = (V_H,E_H)$ be given, for which we would like to determine whether it contains a Hamiltonian cycle.
Set $n=|V_H|$. In the following we will describe the construction of a bipartite graph $G_H$ based on $H$ such that we can determine whether $H$ is Hamiltonian based on knowing the value of the diameter of the perfect matching polytope $P_{G_H}$ associated with $G_H$.

Before we can state the final construction, it will be convenient to first introduce an auxiliary gadget, which we will call a \emph{tower} in the following.
For the definition of a tower, consider any edge $e$ between vertices $v$ and $w$ in some graph.
A \emph{tower $T$ of height $h$ on $e$} is obtained by removing the edge $\{v,w\}$ and replacing it by $2h+2$ vertices $a_i, b_i$ for $i\in \{0, \dots, h\}$ together with edges
\[
E_T = \{\{a_i, b_i\} \colon i \in \{0,\dots, h\}\} \cup \{\{a_i, a_{i+1}\}, \{b_i, b_{i+1}\}\colon i \in \{0,...,h-1\}\} \cup \{\{v,a_0\}, \{w, b_0\}\}.
\]

\begin{figure}[ht]
	\begin{center}
		\scalebox{0.75}{
			\begin{tikzpicture}[scale=2]

				\begin{scope}[every node/.style={thick,draw=black,fill=white,circle,minimum size=17, inner sep=2pt}]
					\node (v0) at (-4,0){};
					\node (w0) at (-2.5,0){};

					\node  (v) at (-0.5,0) {};
					\node  (w) at (1.5,0) {};
				\end{scope}

				\begin{scope}
					\node at (v0){$v$};
					\node at (w0) {$w$};
					\node at (v){$v$};
					\node at (w) {$w$};
				\end{scope}

				\begin{scope}[line width=2pt]
					\draw (v0) -- (w0);
					\draw (v) -- (w);
				\end{scope}

				\tower{0}{0}{3}{}{0}{1}{}

				\draw[very thick, dashed, -stealth] (-2,0) -- (-1,0);
			\end{tikzpicture}
		}
	\end{center}
	\caption{A tower of height $3$ on the edge $\{v,w\}$}
	\label{fig:example-tower}
\end{figure}
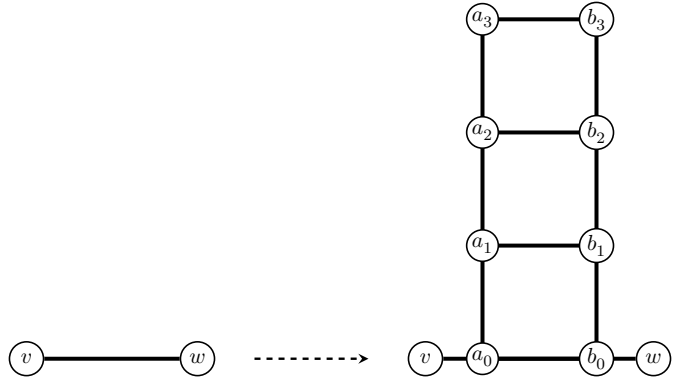

An example of a tower of height $3$ can be found in \cref{fig:example-tower}.
If the tower we consider is clear from the context we use $a_i$ and $b_i$ as above, if there are multiple towers under consideration then we may add a superscript and use $a_i^T$ and $b_i^T$, respectively, to reference the vertices of a tower $T$.

Next we also introduce a gadget to combine multiple towers.
This can be done by iterating the above construction.
In order to build $t$ towers of height $h$ on an edge $e$ we first use the above construction to build a single tower $T_1$.
Once we constructed a total of $k-1$ towers $T_1, \dots, T_{k-1}$ on the edge $e$ we proceed by building a tower gadget on the edge $\{b_0^{T_{k-1}}, w\}$.
\cref{fig:multiple-towers} shows an example of three towers of height 3 constructed on an edge.

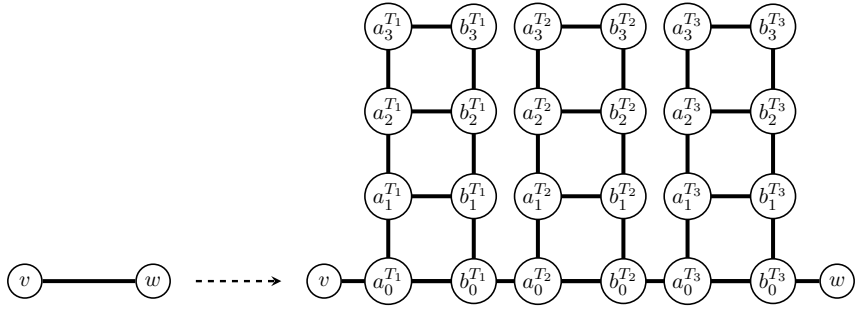
\begin{figure}[ht]
	\begin{center}
		\scalebox{0.75}{
			\begin{tikzpicture}[scale=1.5]

				\begin{scope}[every node/.style={thick,draw=black,fill=white,circle,minimum size=17, inner sep=2pt}]
					\node (v0) at (-3.5,0){};
					\node (w0) at (-2,0){};

					\node  (v) at (0,0) {};
					\node  (w) at (6,0) {};
				\end{scope}

				\begin{scope}
					\node at (v0){$v$};
					\node at (w0) {$w$};
					\node at (v){$v$};
					\node at (w) {$w$};
				\end{scope}

				\begin{scope}[line width=2pt]
					\draw (v0) -- (w0);

					\draw (v) -- (w);

				\end{scope}

				\draw[very thick, dashed, -stealth] (-1.5,0) -- (-0.5,0);

				\edgeWithTowers{0}{0}{6}{0}{3}{3}
			\end{tikzpicture}
		}
		\caption{Three towers $T_1, T_2$ and $T_3$ of height 3 on the edge $\{v,w\}.$}
		\label{fig:multiple-towers}
	\end{center}
\end{figure}

We now construct the auxiliary graph $G_H$ from $H$.
For this let $h> |V_H|$ and $t = 4h$ be numbers we choose precisely later.
As a first step we make the graph $H$ bipartite. To do so duplicate all vertices into pairs, to obtain the following set of new vertices:
\[
V = \{v_1 \colon v\in V_H \} \cup \{v_2 \colon v\in V_H\}.
\]
We will also use the notation $V_1$ for vertices from the first set and $V_2$ for vertices of the second set.
Next we duplicate and redirect all edges of $H$ and add edges between the two copies of every vertex, such that we obtain the following new set of edges:
\[
E = \{\{v_1, v_2\} \colon v \in V_H\} \cup \{\{v_1, w_2\}, \{v_2, w_1\}\colon \{v,w\}\in E_H\}.
\]
Altogether, this yields a bipartite graph $G_H'=(V,E)$ on $2n$ vertices. To go from here to the final graph $G_H$, we take a second step.
Namely, for every edge in $G_H'$ of the form $\{v_1, v_2\}$, i.e., an edge between two copies of the same vertex $v$ of $H$, we add $t$ tower gadgets of height $h$ onto it, as described above, and call the resulting graph $G_H$.
A visualization of this construction can be found in \cref{fig:complete-construction}.

\begin{figure}[ht]
	\centering
	   	\begin{tikzpicture}[scale=1]

   		\pgfdeclarelayer{background}
   		\pgfsetlayers{background,main}
        \newcommand{\displace}{0.5}
		\newcommand{\yoffset}{4}
		\newcommand{\height}{2}

		\pgfmathsetmacro{\wwx}{3}
		\pgfmathsetmacro{\wwy}{-\displace}
		\pgfmathsetmacro{\wx}{-3}
		\pgfmathsetmacro{\wy}{-\displace}

		\pgfmathsetmacro{\uux}{-3 - 0.866*\displace}
		\pgfmathsetmacro{\uuy}{0.5*\displace}
		\pgfmathsetmacro{\ux}{-0.866*\displace}
		\pgfmathsetmacro{\uy}{6*0.866 + 0.5*\displace}

		\pgfmathsetmacro{\vx}{3 + 0.866*\displace}
		\pgfmathsetmacro{\vy}{0.5*\displace}
		\pgfmathsetmacro{\vvx}{0.866*\displace}
		\pgfmathsetmacro{\vvy}{6*0.866 + 0.5*\displace}

		\begin{scope}[every node/.style={thick,draw=black,fill=white,circle,minimum size=17, inner sep=2pt}]
			\node (v0) at (-5.5,\uuy){};
			\node (u0) at (-6.5,\uuy){};
			\node (w0) at (-6, \wy){};

			\node  (v1) at (\vx,\vy) {};
			\node  (v2) at (\vvx,\vvy) {};
			\node  (w2) at (\wwx,\wwy) {};
			\node  (w1) at (\wx,\wy) {};
			\node  (u2) at (\uux,\uuy) {};
			\node  (u1) at (\ux,\uy) {};
		\end{scope}
  
		\pgfmathsetmacro{\arrowy}{(\uuy+\wy)/2}
		\draw[very thick, dashed, -stealth] (-5,\arrowy) -- (-3.9,\arrowy);

		\begin{scope}
			\node at (v0){$v$};
			\node at (w0) {$w$};
			\node at (u0) {$u$};

			\node at (v1){$v_1$};
			\node at (v2){$v_2$};
			\node at (w1){$w_1$};
			\node at (w2){$w_2$};
			\node at (u1){$u_1$};
			\node at (u2) {$u_2$};
		\end{scope}

		\begin{scope}[line width=2pt]
			\draw (v0) -- (w0);
			\draw (u0) -- (w0);
			\draw (u0) -- (v0);

			\draw (v1) -- (v2);
			\draw (w1) -- (w2);
			\draw (u1) -- (u2);

		\end{scope}

		\begin{scope}[line width=2pt]
			\draw (v1) -- (w2);
			\draw (v1) -- (u2);
			\draw (w1) -- (v2);
			\draw (w1) -- (u2);
			\draw (u1) -- (v2);
			\draw (u1) -- (w2);

		\end{scope}

		\setbool{towerlabels}{false}
        \setcounter{towerIndex}{1}

		\edgeWithTowers{\vvx}{\vvy}{\vx}{\vy}{3}{3}
		\edgeWithTowers{\wwx}{\wwy}{\wx}{\wy}{3}{3}
		\edgeWithTowers{\uux}{\uuy}{\ux}{\uy}{3}{3}

	\end{tikzpicture} 	\caption{Illustration of the construction performed in the proof of \cref{thm:perfectmatchingpolytope}.
		We start with the triangle graph on the vertices $\{u,v,w\}$ given on the left. After splitting the nodes, duplicating and redirecting the edges and finally adding the tower gadgets we arrive at the graph $G_H$ on the right.
		For the sake of presentation the height of the towers as well as their number on each edge is reduced compared to the actual construction.}
	\label{fig:complete-construction}
\end{figure}
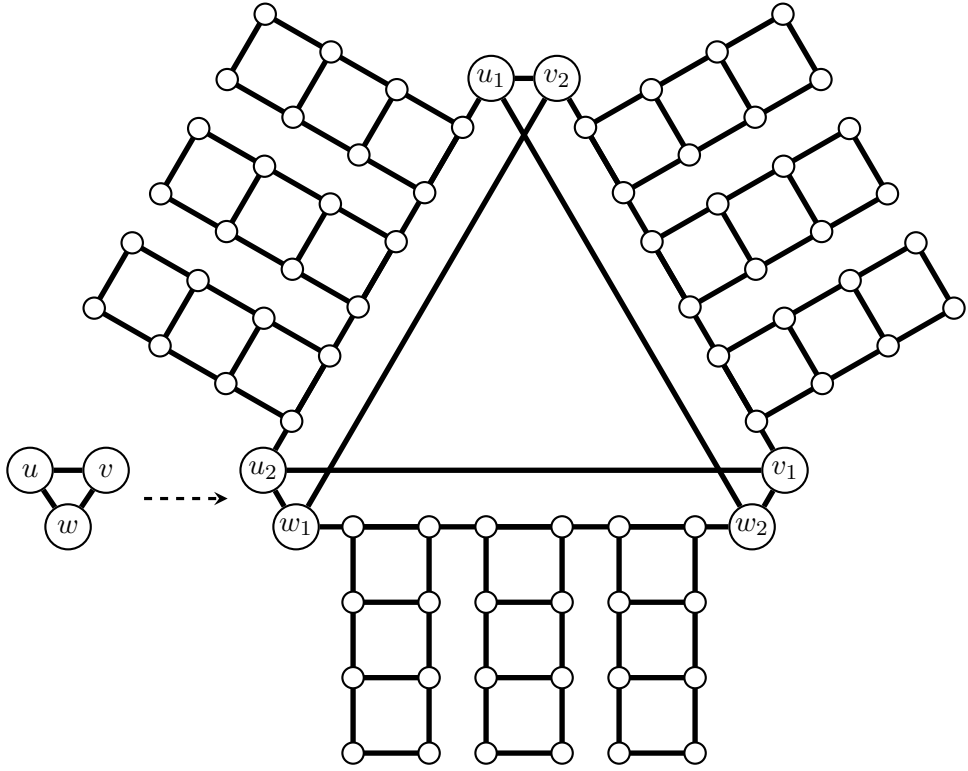

We will obtain the hardness result of \cref{thm:perfectmatchingpolytope} by showing the following bounds on the diameter of the perfect matching polytope of $G_H$.
\begin{theorem}\label{thm:Hamiltonian}
	If $H$ is Hamiltonian, then $\mathrm{diam}(P_{G_H}) \leq 2h + 4n$.
\end{theorem}
\begin{theorem}\label{thm:non-Hamiltonian}
	If $H$ is not Hamiltonian and $|V_H|\ge 3$, then $\mathrm{diam}(P_{G_H})\geq\frac{n}{n-1}(2h-2)$.
\end{theorem}

Assume for now that the two theorems hold, and let us deduce \cref{thm:perfectmatchingpolytope} from it.
To do so, we choose $h$ as $2n^2-n+1=O(n^2)$ and $t=4h=8n^2-4n+4$ in our reduction (which in particular ensures that $G_H$ has polynomial size in terms of $n$, namely its size is in $O(n^5)$). It can easily be checked that with this choice of $h$, we have $2h + 4n< \frac{n}{n-1}(2h-2)$. Indeed, we have
\begin{align*}
	\begin{array}{lrl}
		&2h + 4n &< \frac{n}{n-1}(2h-2)\\
		\Leftrightarrow & 2h(n-1) + 4n(n-1) & < 2hn - 2n\\
		\Leftrightarrow & 4n(n-1) + 2n & < 2h\\
		\Leftrightarrow & 2n^2 - n & < h \enspace .
	\end{array}
\end{align*}

It then follows from \cref{thm:Hamiltonian} and \cref{thm:non-Hamiltonian} that $H$ is Hamiltonian if and only if the diameter of $P_H$ is at most $2h+4n$.

Therefore, we obtain a polynomial reduction from the Hamiltonian cycle problem to the problem of determining the diameter of the perfect matching polytope of a bipartite graph. This then yields the \NP-hardness of the problem of computing the diameter of the perfect matching polytope of a given bipartite graph, and thus the statement of \cref{thm:perfectmatchingpolytope}.

Thus what remains to show are \cref{thm:Hamiltonian,thm:non-Hamiltonian}, which we will establish in \cref{sec:proof-diameter}.
For now we will give a very brief overview of the proofs.
The idea is to investigate the structure of cycles that can change a perfect matching of $G_H$ in multiple towers simultaneously.
Within each tower these cycles take the form of a $v$-$w$ path.
We will show that within one tower any perfect matching can be transformed into any other by consecutively flipping no more than $2h$ paths from $v$ to $w$.
If $H$ contains a Hamiltonian cycle then we can use this cycle to modify a perfect matching in $G_H$ which is "well-aligned" in all towers simultaneously, showing that any two of these "well-aligned" matchings are at distance at most $2h$.
To be more precise, these "well-aligned" perfect matchings are the perfect matchings of $G$ that contain no edge of the form $\{v_1,w_2\}$ for $v,w\in V$, or equivalently the perfect matchings that constitute a perfect matching when restricted to any tower.
To obtain the bound in \cref{thm:Hamiltonian} we will show that any perfect matching in $G_H$ can be transformed into a "well-aligned" one by flipping along at most $2n$ alternating cycles.

For the lower bound of \cref{thm:non-Hamiltonian} we will show that within one tower one can construct two perfect matchings that cannot be transformed into each other using less than $2h-2$ paths between $v$ and $w$.
This construction can be extended to the whole of $G_H$, giving two perfect matchings $M_1$, $M_2$ with the following property: To transform one into the other we need to modify each tower at least $2h-2$ times by a cycle that modifies multiple towers or at least once by a cycle only modifying this tower.
To handle the second case, recall that we add $t = 4h \geq \frac{n}{n-1}(2h-2)$ towers on each edge.
So, if we transform $M_1$ into $M_2$ using less than $t$ cycles, then for each edge there is at least one tower that is modified $2h-2$ times.
Next, we will show that a cycle $C$ in $G_H$ that modifies towers on every edge gives rise to a Hamiltonian cycle in $H$.
Roughly speaking, any such cycle has to contain a path along the subdivision of the $\{v_1, v_2\}$ edge for every $v\in V$, so reverting the constructions gives the desired Hamiltonian cycle.
Combining these facts then yields that for a non-Hamiltonian $H$ the distance between $M_1$ and $M_2$ is at least $\frac{n}{n-1}(2h-2)$, as we have to modify towers on all $n$ edges of the form $\{v_1,v_2\}$ at least $2h-2$ times, while every cycle can only modify towers on at most $n-1$ of those edges, establishing \cref{thm:non-Hamiltonian}.

\subsection{Reduction for \textsc{Monotone Diameter}}\label{sec:reduction-monotone-diameter}

In the following we will present the main ideas of the reduction used to prove \cref{thm:monotoneperfectmatchingpolytope}.
The reduction is motivated by the following precise description of the monotone diameter of the perfect matching polytope of a bipartite graph in graph theoretical terms.

\begin{lemma}\label{lemma:monotone-to-symmetric-difference}
	Let $G$ be a bipartite graph with perfect matching polytope $P_G$.
	Then the monotone diameter of $P_G$ agrees with the maximum number of cycles in the symmetric difference of two perfect matchings in $G$.
\end{lemma}

We postpone the proof of \cref{lemma:monotone-to-symmetric-difference} to \cref{sec:proof-monotone-diameter} and remark that in the special case of complete (bipartite) graphs, this description was already observed by \textcite{Rispoli1992}.

\cref{lemma:monotone-to-symmetric-difference} shows that determining the monotone diameter of $G$ corresponds one to one to a specific cycle packing problem.
A set of pairwise disjoint cycles in a bipartite graph is the symmetric difference of two perfect matchings if and only if the complement of the cycles contains a perfect matching.
So given a bipartite graph $G$, we are interested in packing the maximum number of vertex-disjoint cycles such that the subgraph of $G$ induced by the vertices that are not covered by any of the cycles still contains a perfect matching.
When omitting the latter condition, strong \NP-hardness of the problem was established in \cite{kloks2011results} by a reduction from 3-dimensional matching, however their reduction cannot be used or directly modified to match our variant of the cycle-packing problem with the additional matching condition.

Instead, the idea of our reduction is to observe that the additional matching constraint is automatically guaranteed to be fulfilled if the cycles cover all the vertices.
In particular, in a graph with $|V| = 4k$ we can pack $k$ vertex disjoint cycles such that the uncovered vertices admit a perfect matching if and only if the graph admits $k$ vertex disjoint cycles of length $4$.
Finally, observe that such a collection $\mathcal{C}$ of cycles exists, if and only if there exist two perfect matchings $M_1$ and $M_2$, such that their symmetric difference consists of $k$ cycles.
Indeed, given $\mathcal{C}$ we can define $M_1$ by picking two non-adjacent edges from every cycle and $M_2$ by picking the remaining two edges.
The other way around, given $M_1$ and $M_2$, we get $\mathcal{C}$ as the set of cycles in $M_1\Delta M_2$.

Thus \cref{lemma:monotone-to-symmetric-difference} gives rise to the following consequence.

\begin{corollary}\label{cor:trivial}
	Let $G=(V,E)$ be a bipartite graph. Then we have $\mathrm{mdiam}(P_G)=\frac{|V|}{4}$ if and only if there exists a collection of $4$-cycles in $G$ that covers every vertex in $V$ precisely once.
\end{corollary}

This in particular implies that if one can compute the monotone diameter of perfect matching polytopes, then one can solve the following covering problem.

\begin{mdframed}[innerleftmargin=0.5em, innertopmargin=0.5em, innerrightmargin=0.5em, innerbottommargin=0.5em, userdefinedwidth=0.95\linewidth, align=center]
	{\textsc{Vertex-Disjoint $4$-Cycle Cover}}
	\sloppy

	\noindent
	\textbf{Input:} A bipartite graph $G = (V,E)$.

	\noindent
	\textbf{Decision:} Is there a collection $\mathcal{C}$ of cycles of length $4$ such that every $v\in V$ is part of exactly one cycle of $\mathcal{C}$?
\end{mdframed}

\noindent
In order to prove \cref{thm:monotoneperfectmatchingpolytope} it thus suffices to prove the following.
\begin{lemma}\label{lem:4dm-reduction}
	The \textsc{Vertex-Disjoint $4$-Cycle Cover} problem is \NP-complete.
\end{lemma}

We will prove \cref{lem:4dm-reduction} by giving a reduction from 4-dimensional matching.
The latter is the following decision problem:
\begin{mdframed}[innerleftmargin=0.5em, innertopmargin=0.5em, innerrightmargin=0.5em, innerbottommargin=0.5em, userdefinedwidth=0.95\linewidth, align=center]
	{\textsc{4-Dimensional Matching}}
	\sloppy

	\noindent
	\textbf{Input:} Four disjoint sets $W,X,Y,Z$ and a subset $E \subseteq W\times X \times Y \times Z$.

	\noindent
	\textbf{Decision:} Is there a subset $M\subseteq E$ such that every element of $W, X, Y$ and $Z$ is part of exactly one element of $M$?
\end{mdframed}

\noindent
\textcite{karp2010reducibility} showed that \textsc{3-Dimensional Matching} is \NP-hard. It is straightforward to reduce \textsc{3-Dimensional Matching} to \textsc{4-Dimensional Matching}, and thus also \textsc{4-Dimensional Matching} is \NP-hard.

Let $W,X,Y,Z$ and $E$ be an instance of \textsc{4-Dimensional Matching}.
We will construct an instance of \textsc{Vertex-Disjoint $4$-Cycle Cover}, such that the answers to both problems agree. To do so we construct a bipartite graph $G$ in the following way:
For every element $a\in W\cup X\cup Y\cup Z$ we add a vertex to $G$.
Furthermore, for every element $e \in E$ we add twelve vertices and 28 edges forming a gadget as shown in \cref{fig:4dm-reduction-gadget}.
\begin{figure}
	\centering
	\begin{tikzpicture}[scale=1.2]
        \newcommand{\offset}{0.5}
        \begin{scope}[every node/.style={thick,draw=black,fill=white,circle,minimum size=15, inner sep=1pt}]
                \node[fill=black!20!white] (w) at (-\offset ,-\offset) {$w$};
                \node (x) at (-\offset ,4+\offset) {$x$};
                \node[fill=black!20!white] (y) at (4+\offset ,4+\offset) {$y$};
                \node (z) at (4+\offset ,-\offset) {$z$};
        \end{scope}
        
        \begin{scope}[every node/.style={thick,draw=black,fill=white,circle,minimum size=5, inner sep=1pt}]
                \node (w1) at (1 ,0) {};
                \node[fill=black!20!white] (w2) at (1 ,1) {};
                \node (w3) at (0 ,1) {};
                
                \node[fill=black!20!white] (x1) at (0 ,3) {};
                \node (x2) at (1 ,3) {};
                \node[fill=black!20!white] (x3) at (1 ,4) {};
                
                \node (y1) at (3 ,4) {};
                \node[fill=black!20!white] (y2) at (3 ,3) {};
                \node (y3) at (4 ,3) {};
                
                \node[fill=black!20!white] (z1) at (4 ,1) {};
                \node (z2) at (3 ,1) {};
                \node[fill=black!20!white] (z3) at (3 ,0) {};
        \end{scope}
        
        \begin{scope}[every node/.style={thick,draw=black,fill=white,circle,minimum size=5, inner sep=1pt, font=\scriptsize}]
                \node (w1) at (1 ,0) {$w^e_1$};
                \node[fill=black!20!white] (w2) at (1 ,1) {$w^e_2$};
                \node (w3) at (0 ,1) {$w^e_3$};
                
                \node[fill=black!20!white] (x1) at (0 ,3) {$x^e_1$};
                \node (x2) at (1 ,3) {$x^e_2$};
                \node[fill=black!20!white] (x3) at (1 ,4) {$x^e_3$};
                
                \node (y1) at (3 ,4) {$y^e_1$};
                \node[fill=black!20!white] (y2) at (3 ,3) {$y^e_2$};
                \node (y3) at (4 ,3) {$y^e_3$};
                
                \node[fill=black!20!white] (z1) at (4 ,1) {$z^e_1$};
                \node (z2) at (3 ,1) {$z^e_2$};
                \node[fill=black!20!white] (z3) at (3 ,0) {$z^e_3$};
        \end{scope}

        \begin{scope}[line width=3pt]
            \draw (w) -- (w1);
            \draw (w1) -- (w2);
            \draw (w2) -- (w3);
            \draw (w3) -- (w);

            \draw (x) -- (x1);
            \draw (x1) -- (x2);
            \draw (x2) -- (x3);
            \draw (x3) -- (x);

            \draw (y) -- (y1);
            \draw (y1) -- (y2);
            \draw (y2) -- (y3);
            \draw (y3) -- (y);

            \draw (z) -- (z1);
            \draw (z1) -- (z2);
            \draw (z2) -- (z3);
            \draw (z3) -- (z);
        \end{scope}

        \begin{scope}[line width=3pt, orange!80!red, dashed, dash pattern=on 5pt off 2pt]
            \draw (w1) -- (x1) -- (y1) -- (z1) -- (w1);
            \draw (w2) -- (x2) -- (y2) -- (z2) -- (w2);
            \draw (w3) -- (x3) -- (y3) -- (z3) -- (w3);
        \end{scope}

    \end{tikzpicture} 	\caption{Visualization of the gadget for a hyperedge $e=(x,y,z,w)$, used in the proof of \cref{lem:4dm-reduction}.
		The nodes marked in gray form one side of a bipartition.}
	\label{fig:4dm-reduction-gadget}
\end{figure}
To finish the reduction we will show the following.
\begin{lemma}\label{lem:4d-to-cycle-cover}
	The instance $(W,X,Y,Z,E)$ admits a 4-dimensional matching if and only if the graph $G$ admits a vertex-disjoint cover by cycles of length 4.
\end{lemma}

The idea of the proof is to observe that when covering all vertices of $G$ with vertex disjoint $4$-cycles we either cover all 16, or only the inner 12 vertices of every gadget.
This can then be interpreted as whether the hyperedge corresponding to this gadget is part of the 4-dimensional matching or not.

We postpone a precise description of the reduction, as well as the details of the proof of \cref{lem:4d-to-cycle-cover}, to \cref{sec:proof-monotone-diameter}.

\section{Proof of \cref{thm:perfectmatchingpolytope}}\label{sec:proof-diameter}

In this section we will prove correctness of the reduction presented in \cref{sec:reduction-diameter}.
By the discussion in \cref{sec:reduction-diameter} this boils down to proving \cref{thm:Hamiltonian} and \cref{thm:non-Hamiltonian}.
Before jumping into the proofs we first collect a few definitions and observations.
As we will work with iterative flippings of cycles from now on, it will be useful to introduce the following definition.
\begin{definition}
	Let $G=(V,E)$ be a graph and let $M_1$, $M_2$ be two perfect matchings in $G$.
	We call a collection of cycles $\mathcal{C} = (C_1, \dots, C_\ell)$ in $G$ a \emph{flip sequence of length $\ell$ from $M_1$ to $M_2$} if
	\begin{enumerate}
		\item For every $i\in \{1,\dots, \ell\}$ the cycle $C_i$ is $M_1\Delta C_1\Delta C_2\Delta \dots \Delta C_{i-1}$-alternating, and
		\item $M_2 = M_1\Delta C_1\Delta C_2\Delta \dots \Delta C_{\ell}$.
	\end{enumerate}
\end{definition}

We need two more definitions, that will enable us to closely analyze matchings within one tower.
\begin{definition}
	Let $G=(V,E)$ be a graph and let $e=\{v,w\}\in E$ be an edge of $G$.
	Consider the graph $G_T$ obtained from $G$ by constructing a tower $T$ on $e$.
	We say that a cycle $C$ in $G_T$ \emph{touches} the tower $T$, if $C$ contains the edges $\{v,a_0\}$ and $\{b_0,w\}$.
\end{definition}

In particular a cycle that touches the tower $T$ can be separated into a $v$--$w$ path through the vertices of $T$ and a $w$--$v$ path in $G\setminus e$.

\begin{definition}
	Let $G=(V,E)$ be a graph and let $e=\{v,w\}\in E$ be an edge of $G$.
	Consider the graph $G_T$ obtained from $G$ by constructing a tower $T$ on $e$.
	Let $M$ be a matching in $G_T$.
	Define $\horizontal(M) = \{i\in \{0, \dots, h\}\colon \{a_i, b_i\}\in M\}$, which we call the \emph{horizontal indices} of $M$.
	Furthermore set $d(M) = \min(\horizontal(M))$ which we call the \emph{depth} of $M$.
\end{definition}

As a final bit of preparation we need the following statement on the structural changes of $M$ when augmenting along a $v$-$w$ path:
\begin{lemma}\label{claim:path-properties}
	Let $G=(V,E)$ be a graph and let $e=\{v,w\}\in E$ be an edge of $G$.
	Consider the graph $G_T$ obtained from $G$ by constructing a tower $T$ of height $h$ on $e$.
	Let $M$ be a perfect matching in $G_T$ and let $C$ be an $M$-alternating cycle in $G_T$ that touches $T$.
	Finally set $M' = M\Delta C$ to be the matching we obtain from $M$ by flipping along $C$.
	Then
	\begin{enumerate}
		\item $|\horizontal(M) \Delta \horizontal(M')| = 1$, and
		\item $\horizontal(M)\setminus \horizontal(M') \subseteq \{d(M)\}$.
	\end{enumerate}
	So, in words, flipping the edges of $C$ the number of horizontal edges changes by exactly one and if we remove a horizontal edge of $M$, then it must be the edge $\{a_{d(M)},b_{d(M)}\}$.
\end{lemma}
\begin{proof}
	To prove the first part, let $k$ be the maximum number for which $C$ contains the edge $\{a_{k-1}, a_{k}\}$.
	Then $C$ must contain the edge $\{a_k, b_k\}$ as well.
	As $C$ touches $T$ its restriction to $T$ is a $v$--$w$ path.
	It follows from the structure of $T$ that this restriction consists of the edges $\{v,a_0\}$, $\{b_0,w\}$, $\{a_k,b_k\}$ together with the edges $\{a_{i-1}, a_i\}$ and $\{b_i,b_{i-1}\}$ for $i\in\{1, \dots, k\}$.
	Among these edges, $\{a_k, b_k\}$ is the only horizontal edge.
	Thus it is the unique horizontal edge of $M \Delta M'$.
	In particular, we have $\horizontal(M) \Delta \horizontal(M') = \{k\}$, proving the first statement.

	To prove the second statement, and using the notation from above, it suffices to show that $k\leq d(M)$.
	Assume $k\geq d(M)$, so $C$ contains the vertex $a_{d(M)}$.
	As $C$ is $M$-alternating it must contain the matching edge incident to $a_{d(M)}$, which is the edge $\{a_{d(M)}, b_{d(M)}\}$.
	So $k\geq d(M)$ implies $k = d(M)$ and $k\leq d(M)$ holds as claimed.
\end{proof}

\subsection{Proof of \cref{thm:Hamiltonian}}

In this section we want to show, that if $H$ is Hamiltonian, then the diameter of $P_{G_H}$ is at most $2h+4n$.
By the discussion from \cref{sec:reduction-diameter} this is equivalent to showing that given any two perfect matchings in $G_H$ we can transform one to the other by flipping along at most $2h+4n$ many cycles.
In order to do so, we need to consider how we can extend flip sequences over a tower.
This is summarized in the following technical lemma.

\begin{lemma}\label{lemma:distance-bound}
	Let $G=(V,E)$ be a graph, $e=\{v,w\}\in E$ be an edge of $G$ and $\tilde{M}_1$ and $\tilde{M}_2$ be two perfect matchings in $G$ that contain $e$.
	Let $h\in \mathbb{Z}_{\geq 1}$.
	Furthermore assume we are given a flip sequence $\mathcal{C} = (C_1, \dots, C_{2h})$ of length $2h$ from $\tilde{M}_1$ to $\tilde{M}_2$, such that $e\in C_i$ for all $i\in\{1,\dots, 2h\}$.
	Consider the graph $G_T$ obtained by constructing a tower $T$ of height $h$ on $e$.
	Let $M_1$ and $M_2$ be two perfect matchings in $G_T$ such that for $i\in\{1,2\}$ the matching $M_i$ agrees with $\tilde{M}_i$ outside of the tower gadget $T$.
	In particular the matching $M_i$ contains the edges $\{v, a_0\}$ and $\{b_0, w\}$.

	Then there exists a flip sequence $\mathcal{C'} = (C_1', \dots, C_{2h}')$ of length $2h$ from $M_1$ to $M_2$.
\end{lemma}
\begin{proof}
	The idea is to first construct a local transformation with $v$-$w$ paths in the tower $T$ to transform the restriction of $M_1$ to $T$ to the restriction of $M_2$ to $T$.
	These can afterwards be combined with the cycles of $\mathcal{C}$ to obtain $\mathcal{C}'$.

	As noted in the proof of \cref{claim:path-properties}, every $v$--$w$ path $P$ in the tower $T$ is uniquely defined by a number $k$ such that $\{a_k, b_k\}\in E(P)$.
	Let $P_k$ denote the unique $v$--$w$ path with $\{a_k, b_k\}\in E(P_k)$, i.e.\ $P_k$ is the path with vertex sequence $(v, a_0, a_1, \dots, a_k, b_k, \dots, b_1, b_0, w)$.
	Pause to note that for any perfect matching $N$ in $G_T$ the path $P_{d(N)}$ is $N$-alternating; by definition of $d(N)$.

	Let $\horizontal(M_1)=\{h_1, \dots, h_r\}$ with $h_1<h_2<\dots < h_r$ denote the horizontal indices of $M_1$.
	Note that $h_1 > 0$, as $M_1$ has to match $v$ with $a_0$ and $w$ with $b_0$.
	Consider the collection $\mathcal{P}_1 = (P_{h_1}, \dots, P_{h_r})$.
	By definition we have $h_1 = d(M_1)$ and $P_{h_1}$ is $M_1$-alternating.
	Thus $M_1\Delta P_{h_1}$ is a matching.
	Observe that $d(M_1\Delta P_{h_1}) = h_2$, and hence $P_{h_2}$ is $M_1\Delta P_{h_1}$-alternating.
	Iterating this argument gives $d(M_1\Delta P_{h_1} \Delta \dots \Delta P_{h_{i-1}}) = h_{i-1}$ and thus $P_{h_i}$ is $M_1\Delta P_{h_1} \Delta \dots \Delta P_{h_{i-1}}$-alternating.

	\begin{figure}[ht]
		\centering
		\scalebox{0.75}{
			\begin{tikzpicture}[scale = 0.8]

	\pgfdeclarelayer{background}
	\pgfdeclarelayer{paths}
	\pgfsetlayers{paths, background,main}
    \setcounter{towerIndex}{1}
    \setbool{towerlabels}{false}

    \begin{scope}
    	\coordinate (v) at (-0.3,0);
    	\coordinate (w) at (3.3,0);

    	\begin{scope}[every node/.style={thick,draw=black,fill=white,circle,minimum size=17, inner sep=2pt}]
    		\node at (v) {};
    		\node at (w) {};
    	\end{scope}

        \begin{scope}
            \node at (v){$v$};
            \node at (w) {$w$};
        \end{scope}

        \edgeWithTowers{0}{0}{3}{0}{1}{5}

        \begin{pgfonlayer}{background}
        \begin{scope}[matching]
            \draw (v) -- (1a0);
            \draw (w) -- (1b0);
            \draw (1a1) -- (1b1);
            \draw (1a2) -- (1b2);
            \draw (1a3) -- (1a4);
            \draw (1b3) -- (1b4);
            \draw (1a5) -- (1b5);
        \end{scope}
    \end{pgfonlayer}

    \begin{pgfonlayer}{paths}
    \begin{scope}[black!20!white]
        \flippath{v,1a0,1a1,1b1,1b0,w}{15pt}
    \end{scope}
    \end{pgfonlayer}

    \end{scope}

    \begin{scope}[xshift=5.3cm]
    	\coordinate (v) at (-0.3,0);
    	\coordinate (w) at (3.3,0);

        \begin{scope}[every node/.style={thick,draw=black,fill=white,circle,minimum size=17, inner sep=2pt}]
            \node at (v) {};
            \node at (w) {};
        \end{scope}

        \begin{scope}
            \node at (v){$v$};
            \node at (w) {$w$};
        \end{scope}

        \edgeWithTowers{0}{0}{3}{0}{1}{5}

   \begin{pgfonlayer}{background}

   	\begin{scope}[line width=2pt]
   		\draw (v) -- (w);
   	\end{scope}
        \begin{scope}[matching]
            \draw (2a0) -- (2a1);
            \draw (2b0) -- (2b1);
            \draw (2a2) -- (2b2);
            \draw (2a3) -- (2a4);
            \draw (2b3) -- (2b4);
            \draw (2a5) -- (2b5);
        \end{scope}
    \end{pgfonlayer}

    \begin{pgfonlayer}{paths}
    \begin{scope}[black!20!white]
        \flippath{v,2a0,2a2,2b2,2b0,w}{15pt}
    \end{scope}
    \end{pgfonlayer}

    \end{scope}

    \begin{scope}[xshift=10.6cm]
    	\coordinate (v) at (-0.3,0);
    	\coordinate (w) at (3.3,0);

    	\begin{scope}[every node/.style={thick,draw=black,fill=white,circle,minimum size=17, inner sep=2pt}]
    		\node at (v) {};
    		\node at (w) {};
    	\end{scope}

        \begin{scope}
            \node at (v){$v$};
            \node at (w) {$w$};
        \end{scope}

        \edgeWithTowers{0}{0}{3}{0}{1}{5}

        \begin{pgfonlayer}{background}
        \begin{scope}[matching]
            \draw (v) -- (3a0);
            \draw (w) -- (3b0);
            \draw (3a2) -- (3a1);
            \draw (3b2) -- (3b1);
            \draw (3a3) -- (3a4);
            \draw (3b3) -- (3b4);
            \draw (3a5) -- (3b5);
        \end{scope}
    \end{pgfonlayer}

    \begin{pgfonlayer}{paths}
    \begin{scope}[black!20!white]
        \flippath{v,3a0,3a5,3b5,3b0,w}{15pt}
    \end{scope}
    \end{pgfonlayer}

    \end{scope}

    \begin{scope}[xshift=15.9cm]
    	\coordinate (v) at (-0.3,0);
    	\coordinate (w) at (3.3,0);

    	\begin{scope}[every node/.style={thick,draw=black,fill=white,circle,minimum size=17, inner sep=2pt}]
    		\node at (v) {};
    		\node at (w) {};
    	\end{scope}

        \begin{scope}
            \node at (v){$v$};
            \node at (w) {$w$};
        \end{scope}

        \edgeWithTowers{0}{0}{3}{0}{1}{5}

        \begin{pgfonlayer}{background}

        	\begin{scope}[line width=2pt]
        		\draw (v) -- (w);
        	\end{scope}

        	\begin{scope}[matching]
        		\draw (4a0) -- (4a1);
        		\draw (4b0) -- (4b1);
        		\draw (4a2) -- (4a3);
        		\draw (4b2) -- (4b3);
        		\draw (4a4) -- (4a5);
        		\draw (4b4) -- (4b5);
        	\end{scope}
        \end{pgfonlayer}
    \end{scope}
\end{tikzpicture} 		}
		\caption{Example of the transformation from \cref{lemma:distance-bound}.
			We transform an arbitrary matching in a tower $T$ to the unique matching that covers all $a_i,b_i$ using no horizontal edges.
			In each step we mark the alternating $v$-$w$ path we use.}
		\label{fig:enter-label}
	\end{figure}

	Altogether we get that the restriction $N$ of $M_1\Delta P_{h_1} \Delta \dots \Delta P_{h_{r}}$ to $T$ is a matching covering the vertices $a_i$ and $b_i$ for all $i\in \{0, \dots, h\}$ that contains no horizontal edge.
	In particular $N$ does not contain $\{a_h,b_h\}$.
	Thus, it has to match $a_h$ and $b_h$ to $a_{h-1}$ and $b_{h-1}$, respectively.
	Repeating this argument shows that every such matching has to match $a_{h-2i}$ and $b_{h-2i}$ to $a_{h-2i-1}$ and $b_{h-2i-1}$, respectively, for $i \in \left\{0, \dots, \left\lfloor\frac{h-1}{2}\right\rfloor\right\}$.

	Doing likewise for $M_2$ gives rise to a collection $\mathcal{P}_2$ of paths, which transforms $M_2$ to a perfect matching whose restriction to $T$ contains no horizontal edges.
	By the above argument, this property uniquely defines the restriction of perfect matchings to $T$, and thus this restriction must coincide with $N$.

	Let $\mathcal{P}$ be the collection of paths consisting of the paths of $\mathcal{P}_1$, followed by the paths of $\mathcal{P}_2$ in reverse order.
	Note that $|\mathcal{P}| = |\horizontal(M_1)|+|\horizontal(M_2)|\leq 2h$.
	Furthermore when flipping all the paths of $\mathcal{P}$ the restriction of $M_1$ to $T$ is transformed to the restriction of $M_2$ to $T$.
	In order to get the correct number of paths, we append to $\mathcal{P}$ the path $P=(v,a_0, b_0, w)$ a number of times such that we end up with $2h$ paths in total.
	First we need to check that the path $P$ is alternating with respect to the restriction of $M_2$ to $T$.
	This is due to the fact that $M_2$ contains the edges $\{v, a_0\}$ and $\{b_0, w\}$.
	After flipping the edges of an alternating path this path remains alternating, so appending the copies of $P$ still gives a sequence of alternating paths.

	In order to get the desired collection $\mathcal{C'}$ we define $C_i'$ to be the cycle in $G_T$ we obtain from $C_i$ by replacing the edge $e$ by the path $P_i$.
	By considering $\mathcal{C}'$ separately inside of $T$ and in the rest of $G_T$ we can see that it indeed transforms $M_1$ into $M_2$.
	Inside of $T$ this follows from how we constructed $\mathcal{P}$, outside of $T$ this follows from the assumptions on $\mathcal{C}$.
	Importantly, every cycle in $\mathcal{C}'$ (by assumption on $\mathcal{C}$ and by construction) contains both the edges $\{v,a_0\}$ and $\{w,b_0\}$. Since $\mathcal{C}'$ has even size, these edges get flipped an even  number of times and thus stay in the matching when transforming $M_1$ using the flip-sequence $\mathcal{C}'$.
\end{proof}

\begin{remark}
	One can strengthen the result above slightly and show, that $2h-2$ cycles suffice, matching the lower bound we will show later in \cref{lemma:far-away-matchings}.
\end{remark}

\begin{figure}
	\begin{subfigure}[t]{0.5\textwidth}
		\centering
		\scalebox{0.6}{
			\begin{tikzpicture}[scale=1]

	\pgfdeclarelayer{background}
	\pgfdeclarelayer{paths}
	\pgfsetlayers{paths, background,main}

	\begin{scope}[every node/.style={thick,draw=black,fill=white,circle,minimum size=17, inner sep=2pt}]
		\node (u) at (-3,0){};
		\node (v) at (3,0){};
		\node (w) at (0, -0.866*6){};
	\end{scope}
	\node at (0, -0.866*6 -2){};

	\begin{scope}
		\node at (u){$u$};
		\node at (v){$v$};
		\node at (w){$w$};
	\end{scope}

	\begin{pgfonlayer}{background}
		\begin{scope}[line width=2pt]
			\draw (u) -- (v);
			\draw (v) -- (w);
			\draw (w) -- (u);
		\end{scope}
	\end{pgfonlayer}

\begin{pgfonlayer}{paths}
	\begin{scope}[black!20!white]
		\flippath{u,v,w,u}{12pt}
	\end{scope}
\end{pgfonlayer}

\end{tikzpicture} 		}
		\caption{The graph $H$ with a Hamiltonian cycle $C$ marked.}
		\label{subfig:hamiltonian_example_0}
	\end{subfigure}
	\begin{subfigure}[t]{0.5\textwidth}
		\centering
		\scalebox{0.6}{
			\begin{tikzpicture}[scale=1]

	\pgfdeclarelayer{background}
	\pgfdeclarelayer{paths}
	\pgfsetlayers{paths, background,main}
    \newcommand{\displace}{0.5}
	\newcommand{\yoffset}{4}
	\newcommand{\height}{2}

	\pgfmathsetmacro{\wwx}{3}
	\pgfmathsetmacro{\wwy}{-\displace}
	\pgfmathsetmacro{\wx}{-3}
	\pgfmathsetmacro{\wy}{-\displace}

	\pgfmathsetmacro{\uux}{-3 - 0.866*\displace}
	\pgfmathsetmacro{\uuy}{0.5*\displace}
	\pgfmathsetmacro{\ux}{-0.866*\displace}
	\pgfmathsetmacro{\uy}{6*0.866 + 0.5*\displace}

	\pgfmathsetmacro{\vx}{3 + 0.866*\displace}
	\pgfmathsetmacro{\vy}{0.5*\displace}
	\pgfmathsetmacro{\vvx}{0.866*\displace}
	\pgfmathsetmacro{\vvy}{6*0.866 + 0.5*\displace}

	\begin{scope}[every node/.style={thick,draw=black,fill=white,circle,minimum size=17, inner sep=2pt}]

		\node  (v1) at (\vx,\vy) {};
		\node  (v2) at (\vvx,\vvy) {};
		\node  (w2) at (\wwx,\wwy) {};
		\node  (w1) at (\wx,\wy) {};
		\node  (u2) at (\uux,\uuy) {};
		\node  (u1) at (\ux,\uy) {};
	\end{scope}

	\begin{scope}
		\node at (v1){$v_1$};
		\node at (v2){$v_2$};
		\node at (w1){$w_1$};
		\node at (w2){$w_2$};
		\node at (u1){$u_1$};
		\node at (u2) {$u_2$};
	\end{scope}

\begin{pgfonlayer}{background}
	\begin{scope}[line width=2pt]
		\draw (v1) -- (v2);
		\draw (w1) -- (w2);
		\draw (u1) -- (u2);
	\end{scope}

	\begin{scope}[line width=2pt]
		\draw (v1) -- (w2);
		\draw (v1) -- (u2);
		\draw (w1) -- (v2);
		\draw (w1) -- (u2);
		\draw (u1) -- (v2);
		\draw (u1) -- (w2);
	\end{scope}
\end{pgfonlayer}

	\setbool{towerlabels}{false}
    \setcounter{towerIndex}{1}

	\edgeWithTowers{\vvx}{\vvy}{\vx}{\vy}{3}{3}
	\edgeWithTowers{\wwx}{\wwy}{\wx}{\wy}{3}{3}
	\edgeWithTowers{\uux}{\uuy}{\ux}{\uy}{3}{3}

	\begin{pgfonlayer}{background}

	\begin{scope}[matching]
		\draw (v2) -- (1a0);
		\draw (1b0) -- (2a0);
		\draw (2b0) -- (3a0);
		\draw (3b0) -- (v1);

		\draw (w2) -- (4a0);
		\draw (4b0) -- (5a0);
		\draw (5b0) -- (6a0);
		\draw (6b0) -- (w1);

		\draw (u2) -- (7a0);
		\draw (7b0) -- (8a0);
		\draw (8b0) -- (9a0);
		\draw (9b0) -- (u1);

		\draw (1a1) -- (1b1);
		\draw (1a2) -- (1b2);
		\draw (1a3) -- (1b3);

		\draw (2a1) -- (2a2);
		\draw (2b1) -- (2b2);
		\draw (2a3) -- (2b3);

		\draw (3a1) -- (3b1);
		\draw (3a2) -- (3a3);
		\draw (3b2) -- (3b3);

		\draw (4a1) -- (4b1);
		\draw (4a2) -- (4b2);
		\draw (4a3) -- (4b3);

		\draw (5a1) -- (5b1);
		\draw (5a2) -- (5a3);
		\draw (5b2) -- (5b3);

		\draw (6a1) -- (6b1);
		\draw (6a2) -- (6a3);
		\draw (6b2) -- (6b3);

		\draw (7a1) -- (7a2);
		\draw (7b1) -- (7b2);
		\draw (7a3) -- (7b3);

		\draw (8a1) -- (8b1);
		\draw (8a2) -- (8b2);
		\draw (8a3) -- (8b3);

		\draw (9a1) -- (9b1);
		\draw (9a2) -- (9b2);
		\draw (9a3) -- (9b3);

	\end{scope}
\end{pgfonlayer}

\begin{pgfonlayer}{paths}
	\begin{scope}[black!20!white]
		\flippath{v2,1a0,1a1,1b1,1b0,2a0,2a3,2b3,2b0,3a0,3a0,3b0,3b0,v1,
				  w2,4a0,4a1,4b1,4b0,5a0,5a1,5b1,5b0,6a0,6a0,6b0,6b0,w1,
				  u2,7a0,7a3,7b3,7b0,8a0,8a0,8b0,8b0,9a0,9a1,9b1,9b0,u1,
				  v2}{10pt}
	\end{scope}
\end{pgfonlayer}

\end{tikzpicture} 		}
		\caption{A matching $M$ in $G_H$. Using $C$, one can find a cycle $C_1$, marked in gray, that goes through every tower.}
		\label{subfig:hamiltonian_example_1}
	\end{subfigure}
	\begin{subfigure}[t]{0.5\textwidth}
		\centering
		\scalebox{0.6}{
			\begin{tikzpicture}[scale=1]

	\pgfdeclarelayer{background}
	\pgfdeclarelayer{paths}
	\pgfsetlayers{paths, background,main}
    \newcommand{\displace}{0.5}
	\newcommand{\yoffset}{4}
	\newcommand{\height}{2}

	\pgfmathsetmacro{\wwx}{3}
	\pgfmathsetmacro{\wwy}{-\displace}
	\pgfmathsetmacro{\wx}{-3}
	\pgfmathsetmacro{\wy}{-\displace}

	\pgfmathsetmacro{\uux}{-3 - 0.866*\displace}
	\pgfmathsetmacro{\uuy}{0.5*\displace}
	\pgfmathsetmacro{\ux}{-0.866*\displace}
	\pgfmathsetmacro{\uy}{6*0.866 + 0.5*\displace}

	\pgfmathsetmacro{\vx}{3 + 0.866*\displace}
	\pgfmathsetmacro{\vy}{0.5*\displace}
	\pgfmathsetmacro{\vvx}{0.866*\displace}
	\pgfmathsetmacro{\vvy}{6*0.866 + 0.5*\displace}

	\begin{scope}[every node/.style={thick,draw=black,fill=white,circle,minimum size=17, inner sep=2pt}]

		\node  (v1) at (\vx,\vy) {};
		\node  (v2) at (\vvx,\vvy) {};
		\node  (w2) at (\wwx,\wwy) {};
		\node  (w1) at (\wx,\wy) {};
		\node  (u2) at (\uux,\uuy) {};
		\node  (u1) at (\ux,\uy) {};
	\end{scope}

	\begin{scope}
		\node at (v1){$v_1$};
		\node at (v2){$v_2$};
		\node at (w1){$w_1$};
		\node at (w2){$w_2$};
		\node at (u1){$u_1$};
		\node at (u2) {$u_2$};
	\end{scope}

\begin{pgfonlayer}{background}
	\begin{scope}[line width=2pt]
		\draw (v1) -- (v2);
		\draw (w1) -- (w2);
		\draw (u1) -- (u2);
	\end{scope}

	\begin{scope}[line width=2pt]
		\draw (v1) -- (w2);
		\draw (v1) -- (u2);
		\draw (w1) -- (v2);
		\draw (w1) -- (u2);
		\draw (u1) -- (v2);
		\draw (u1) -- (w2);
	\end{scope}
\end{pgfonlayer}

	\setbool{towerlabels}{false}
    \setcounter{towerIndex}{1}

	\edgeWithTowers{\vvx}{\vvy}{\vx}{\vy}{3}{3}
	\edgeWithTowers{\wwx}{\wwy}{\wx}{\wy}{3}{3}
	\edgeWithTowers{\uux}{\uuy}{\ux}{\uy}{3}{3}

	\begin{pgfonlayer}{background}

	\begin{scope}[matching]
		\draw (v2) -- (u1);
		\draw (u2) -- (w1);
		\draw (w2) -- (v1);

		\draw (1a0) -- (1a1);
		\draw (1b0) -- (1b1);
		\draw (1a2) -- (1b2);
		\draw (1a3) -- (1b3);

		\draw (2a1) -- (2a0);
		\draw (2b1) -- (2b0);
		\draw (2a3) -- (2a2);
		\draw (2b2) -- (2b3);

		\draw (3a0) -- (3b0);
		\draw (3a1) -- (3b1);
		\draw (3a2) -- (3a3);
		\draw (3b2) -- (3b3);

		\draw (4a0) -- (4a1);
		\draw (4b0) -- (4b1);
		\draw (4a2) -- (4b2);
		\draw (4a3) -- (4b3);

		\draw (5a0) -- (5a1);
		\draw (5b0) -- (5b1);
		\draw (5a2) -- (5a3);
		\draw (5b2) -- (5b3);

		\draw (6a1) -- (6b1);
		\draw (6a2) -- (6a3);
		\draw (6b2) -- (6b3);
		\draw (6a0) -- (6b0);

		\draw (7a1) -- (7a0);
		\draw (7b1) -- (7b0);
		\draw (7a2) -- (7a3);
		\draw (7b2) -- (7b3);

		\draw (8a1) -- (8b1);
		\draw (8a2) -- (8b2);
		\draw (8a3) -- (8b3);
		\draw (8a0) -- (8b0);

		\draw (9a0) -- (9a1);
		\draw (9b0) -- (9b1);
		\draw (9a2) -- (9b2);
		\draw (9a3) -- (9b3);

	\end{scope}
\end{pgfonlayer}

\begin{pgfonlayer}{paths}
	\begin{scope}[black!20!white]
		\flippath{v2,1a0,1a2,1b2,1b0,2a0,2a1,2b1,2b0,3a0,3a0,3b0,3b0,v1,
				  w2,4a0,4a2,4b2,4b0,5a0,5a3,5b3,5b0,6a0,6a0,6b0,6b0,w1,
				  u2,7a0,7a3,7b3,7b0,8a0,8a0,8b0,8b0,9a0,9a2,9b2,9b0,u1,
				  v2}{10pt}
	\end{scope}
\end{pgfonlayer}

\end{tikzpicture} 		}
		\caption{Flipping $C_1$ we get a new matching, $M_1$, that may differ from $M$ in every tower. Using $C$ again we can find a new cycle $C_2$, again marked in gray.}
		\label{subfig:hamiltonian_example_2}
	\end{subfigure}
	\begin{subfigure}[t]{0.5\textwidth}
		\centering
		\scalebox{0.6}{
			\begin{tikzpicture}[scale=1]

	\pgfdeclarelayer{background}
	\pgfdeclarelayer{paths}
	\pgfsetlayers{paths, background,main}
    \newcommand{\displace}{0.5}
	\newcommand{\yoffset}{4}
	\newcommand{\height}{2}

	\pgfmathsetmacro{\wwx}{3}
	\pgfmathsetmacro{\wwy}{-\displace}
	\pgfmathsetmacro{\wx}{-3}
	\pgfmathsetmacro{\wy}{-\displace}

	\pgfmathsetmacro{\uux}{-3 - 0.866*\displace}
	\pgfmathsetmacro{\uuy}{0.5*\displace}
	\pgfmathsetmacro{\ux}{-0.866*\displace}
	\pgfmathsetmacro{\uy}{6*0.866 + 0.5*\displace}

	\pgfmathsetmacro{\vx}{3 + 0.866*\displace}
	\pgfmathsetmacro{\vy}{0.5*\displace}
	\pgfmathsetmacro{\vvx}{0.866*\displace}
	\pgfmathsetmacro{\vvy}{6*0.866 + 0.5*\displace}

	\begin{scope}[every node/.style={thick,draw=black,fill=white,circle,minimum size=17, inner sep=2pt}]

		\node  (v1) at (\vx,\vy) {};
		\node  (v2) at (\vvx,\vvy) {};
		\node  (w2) at (\wwx,\wwy) {};
		\node  (w1) at (\wx,\wy) {};
		\node  (u2) at (\uux,\uuy) {};
		\node  (u1) at (\ux,\uy) {};
	\end{scope}

	\begin{scope}
		\node at (v1){$v_1$};
		\node at (v2){$v_2$};
		\node at (w1){$w_1$};
		\node at (w2){$w_2$};
		\node at (u1){$u_1$};
		\node at (u2) {$u_2$};
	\end{scope}

\begin{pgfonlayer}{background}
	\begin{scope}[line width=2pt]
		\draw (v1) -- (v2);
		\draw (w1) -- (w2);
		\draw (u1) -- (u2);
	\end{scope}

	\begin{scope}[line width=2pt]
		\draw (v1) -- (w2);
		\draw (v1) -- (u2);
		\draw (w1) -- (v2);
		\draw (w1) -- (u2);
		\draw (u1) -- (v2);
		\draw (u1) -- (w2);
	\end{scope}
\end{pgfonlayer}

	\setbool{towerlabels}{false}
    \setcounter{towerIndex}{1}

	\edgeWithTowers{\vvx}{\vvy}{\vx}{\vy}{3}{3}
	\edgeWithTowers{\wwx}{\wwy}{\wx}{\wy}{3}{3}
	\edgeWithTowers{\uux}{\uuy}{\ux}{\uy}{3}{3}

	\begin{pgfonlayer}{background}

	\begin{scope}[matching]
		\draw (v2) -- (1a0);
		\draw (1b0) -- (2a0);
		\draw (2b0) -- (3a0);
		\draw (3b0) -- (v1);

		\draw (w2) -- (4a0);
		\draw (4b0) -- (5a0);
		\draw (5b0) -- (6a0);
		\draw (6b0) -- (w1);

		\draw (u2) -- (7a0);
		\draw (7b0) -- (8a0);
		\draw (8b0) -- (9a0);
		\draw (9b0) -- (u1);

		\draw (1a1) -- (1a2);
		\draw (1b1) -- (1b2);
		\draw (1a3) -- (1b3);

		\draw (2a1) -- (2b1);
		\draw (2a2) -- (2a3);
		\draw (2b2) -- (2b3);

		\draw (3a1) -- (3b1);
		\draw (3a2) -- (3a3);
		\draw (3b2) -- (3b3);

		\draw (4a1) -- (4a2);
		\draw (4b1) -- (4b2);
		\draw (4a3) -- (4b3);

		\draw (5a1) -- (5a2);
		\draw (5b1) -- (5b2);
		\draw (5a3) -- (5b3);

		\draw (6a1) -- (6b1);
		\draw (6a2) -- (6a3);
		\draw (6b2) -- (6b3);

		\draw (7a1) -- (7a2);
		\draw (7b1) -- (7b2);
		\draw (7a3) -- (7b3);

		\draw (8a1) -- (8b1);
		\draw (8a2) -- (8b2);
		\draw (8a3) -- (8b3);

		\draw (9a1) -- (9a2);
		\draw (9b1) -- (9b2);
		\draw (9a3) -- (9b3);

	\end{scope}
\end{pgfonlayer}

\end{tikzpicture} 		}
		\caption{Flipping $C_2$ we can again modify $M_1$ in every tower.}
		\label{subfig:hamiltonian_example_3}
	\end{subfigure}
	\caption{Visualization of the idea used in the construction in the proof of \cref{thm:Hamiltonian}.
		The number and height of the towers, the matchings and the augmenting paths may differ from the actual values used in the proof.
		The graph $H$ and the auxiliary graph $G_H$ coincide with the graphs in \cref{fig:complete-construction}.
		The Hamiltonian cycle $C$ in $H$, see \cref{subfig:hamiltonian_example_0}, can be used to modify a matching in all towers simultaneously.
		Note, that the matchings in \cref{subfig:hamiltonian_example_1} and \cref{subfig:hamiltonian_example_3} contain every second edge on the subdivisions of the edges $v_1, v_2$ for $v\in V_H$, while the matching in \cref{subfig:hamiltonian_example_2} instead uses edges of the form $\{v_1, w_2\}$ for $\{v,w\}\in C$.
	}
	\label{fig:visualization_hamiltonian}
\end{figure}

Equipped with the above lemma we are set to finish the proof of \cref{thm:Hamiltonian}.

\bigskip
\par\noindent
\emph{Proof of \cref{thm:Hamiltonian}.}\quad
Let $M_1$ and $M_2$ be two matchings in $G_H$.

First we want to transform $M_1$ and $M_2$ to matchings for which we can apply \cref{lemma:distance-bound}.
To do so consider an auxiliary matching $M$ as follows.
On every edge $\{v_1, v_2\}$, which we subdivided $2t$-times when constructing the towers, every other edge is part of $M$, including the edges incident to $v_1$ and $v_2$.
So if the towers on $\{v_1,v_2\}$ are $(T_1, \dots, T_t)$ in this order, then the edges $\{v_1, a_0^{T_1}\}$, $\{b_0^{T_t}, w\}$ as well as the edges $\{b_0^{T_{i}}, a_0^{T_{i+1}}\}$ for $i\in \{1, \dots, t-1\}$ are part of $M$.
We extend $M$ to a perfect matching by an arbitrary matching in every tower $T$, e.g.\ by matching $a_i^T$ with $b_i^T$ for all $i\in \{1,\dots, h\}$.

Now consider the symmetric difference $M_1 \Delta M$.
It can be written as a union $\mathcal{C}_{\text{all}}$ of $M_1$-alternating cycles.
Let $\mathcal{C}$ be the set of cycles of $\mathcal{C}_{\text{all}}$, that contain a vertex of $V_1$ or $V_2$.
Then we have $|\mathcal{C}|\leq |V_1|+|V_2| = 2n$.
Let $N_1$ be the matching we obtain from $M_1$ by flipping the cycles of $\mathcal{C}$.
Then $N_1$ matches the vertices of $V_1$ and $V_2$ in the same way as $M$ does.
In particular considering an edge $\{v_1,v_2\}$ on which we added $t$ towers $(T_1, \dots, T_t)$ in this order, then the edges $\{v_1, a_0^{T_1}\}$, $\{b_0^{T_t}, w\}$ are part of $N_1$.
Using the structure of the towers, in particular that they contain an even number of vertices, we can deduce that $N_1$ contains the edges $\{b_0^{T_{i}}, a_0^{T_{i+1}}\}$ for $i\in \{1, \dots, t-1\}$.

In the same way as described above for $M_1$, we can also find a perfect matching $N_2$ that is obtained from $M_2$ by flipping at most $2n$ alternating cycles and such that $N_2$ matches the vertices in $V_1 \cup V_2$ as well as all the vertices $a_0^T, b_0^T$ for all towers $T$ in the same way that $M$ does.

Let $C=(v_1, \dots , v_n, v_1)$ be the vertex order of a Hamiltonian cycle in $H$.
Consider the graph $G_H'$, which we recall was obtained from $H$ by splitting every node into two copies.
In $G_H'$ we define the cycle $C'=(v_1^1, v_1^2, v_2^1, v_2^2,\dots, v_n^2, v_1^1)$ and consider the cycle collection $\mathcal{C}$ consisting of $2h$ copies of $C'$.

Let us assume we obtained $G_H$ from $G_H'$ by adding the towers $T_1, \dots, T_k$ in this order.
Let $N_1^j$ and $N_2^j$ be the matchings we obtain from $N_1$ and $N_2$, respectively, by replacing the towers $T_{j+1}, \dots, T_k$ by a single matching edge.
Due to the construction above this gives perfect matchings in the graph we obtain from $G_H'$ by constructing the towers $T_1,\dots, T_j$.

Inductively we will show, that there is a flip sequence of length $\ell$ from $N_1^j$ to $N_2^j$ of length $2h$.
For $j=0$ the matchings $N_1^0$ and $N_2^0$ agree.
Namely, by construction of the auxiliary matching $M$ they are both the matching in $G_H'$ that consists of the edges $\{v^1,v^2\}$ for $v\in H$.
To conclude assume that we have a flip sequence $\mathcal{C}_j$ from $N_1^j$ to $N_2^j$.
Then we can apply \cref{lemma:distance-bound} to the tower $T_{j+1}$ to extend it to a flip sequence $\mathcal{C}_{j+1}$ from $N_1^{j+1}$ to $N_2^{j+1}$.

In particular combining this with the flip sequence from $M_1$ to $N_1$ and from $M_2$ to $N_2$ (in reverse) we proved that we can reach $M_2$ from $N_1$ by flipping $2h + 4n$ many cycles.
As $M_1$ and $M_2$ were arbitrary perfect matchings this shows that any two vertices of $P_H$ have distance at most $2h+4n$, finishing the proof of \cref{thm:Hamiltonian}.
\qed

\subsection{Proof of \cref{thm:non-Hamiltonian}}

In order to prove \cref{thm:non-Hamiltonian}, we have to prove that if $H$ is not Hamiltonian, then the diameter of $P_{G_H}$ is relatively large, namely at least $\frac{n}{n-1}(2h-2)$.
We will first consider a single tower $T$.
In $T$ we will construct two matchings, such that every transformation from one to the other using $v$--$w$ paths only, requires $2h-2$ many paths.
We will use this to finally construct two matchings in $G_H$, such that every flip sequence from one to the other has length at least $\frac{n}{n-1}(2h-2)$.
This is collected in the following technical Lemma.

\begin{lemma}\label{lemma:far-away-matchings}
	Let $G=(V,E)$ be a graph, $e=\{v,w\}\in E$ be an edge of $G$ and $\tilde{M}_1$ and $\tilde{M}_2$ be two perfect matchings in $G$ that both contain $e$.
	Consider the graph $G_T$ obtained by constructing a tower $T$ of height $h$ on $e$.
	Then there exist two perfect matchings $M_1$ and $M_2$ in $G_T$, such that $M_i$ agrees with $\tilde{M}_i$ outside of the tower $T$, with the following property: For every flip sequence $\mathcal{C} = (C_1, \dots, C_\ell)$ from $M_1$ to $M_2$, we have that $\mathcal{C}$ contains a cycle that is fully contained in $T$ or $\mathcal{C}$ contains at least $2h-2$ cycles touching the tower $T$.
\end{lemma}
\begin{proof}
	We will explicitly construct two such matchings $M_1$ and $M_2$.
	For all edges $f$ of $G_T$ that are not connected to the tower $T$ (these are the edges in $E\setminus \{e\}$) by assumption the matchings have to agree with $\tilde{M}_1$ and $\tilde{M}_2$, respectively.
	Within the tower $T$ let $M_1$ contain the edges $\{a_i, b_i\}$ for $i\in \{1, \dots, h\}$, together with the edges $\{v,a_0\}$ and $\{b_0, w\}$.

	Let the second matching $M_2$ coincide with $M_1$ within the tower $T$, except that it contains the edges $\{a_{h-1}, a_h\}$ and $\{b_{h-1}, b_h\}$ instead of $\{a_{h-1}, b_{h-1}\}$ and $\{a_h, b_h\}$, so
	\begin{align*}
		M_1 &= (\tilde{M_1} \setminus \{e\}) \cup \{\{v, a_0\}, \{b_0, w\}\} \cup \{\{a_i, b_i\}\colon i\in [h]\} \\
		M_2 &= (\tilde{M_2} \setminus \{e\}) \cup \{\{v, a_0\}, \{b_0, w\}\} \cup \{\{a_i, b_i\}\colon i\in [h-2]\} \cup  \left\{\{b_{h-1}, b_h\}, \{a_h, a_{h-1}\}\right\}.
	\end{align*}
	This defines perfect matchings in $G_T$, as we assumed $\tilde{M}_1$ and $\tilde{M_2}$ to contain the edge $e$.
	A visualization of $M_1$ and $M_2$ for a tower of height $5$ can be seen in \cref{fig:far-matchings-1} and \cref{fig:far-matchings-2}.

	\begin{figure}[ht]
		\centering
		\begin{subfigure}{0.45\textwidth}
			\centering
			\scalebox{0.5}{
				\begin{tikzpicture}
    \setcounter{towerIndex}{1}
    \setbool{towerlabels}{false}
    \pgfdeclarelayer{background}
    \pgfsetlayers{background,main}

    \begin{scope}[every node/.style={thick,draw=black,fill=white,circle,minimum size=17, inner sep=2pt}]
        \node  (v) at (0,0) {};
        \node  (w) at (3,0) {};
    \end{scope}

    \begin{scope}
        \node at (v){$v$};
        \node at (w) {$w$};
    \end{scope}

    \edgeWithTowers{0}{0}{3}{0}{1}{5}

	\begin{pgfonlayer}{background}
		\begin{scope}[line width=2pt]
			\draw (v) -- (w);

		\end{scope}
    \begin{scope}[matching]
        \draw (v) -- (1a0);
        \draw (w) -- (1b0);
        \draw (1a1) -- (1b1);
        \draw (1a2) -- (1b2);
        \draw (1a3) -- (1b3);
        \draw (1a4) -- (1b4);
        \draw (1a5) -- (1b5);
    \end{scope}
\end{pgfonlayer}
\end{tikzpicture} 			}
			\caption{The restriction of the matching $M_1$ to the tower $T$.}
			\label{fig:far-matchings-1}
		\end{subfigure}
		\begin{subfigure}{0.45\textwidth}
			\centering
			\scalebox{0.5}{
				\begin{tikzpicture}
	\pgfdeclarelayer{background}
	\pgfsetlayers{background,main}
    \setcounter{towerIndex}{1}
    \setbool{towerlabels}{false}
    \begin{scope}[every node/.style={thick,draw=black,fill=white,circle,minimum size=17, inner sep=2pt}]
  			\node  (v) at (0,0) {};
  			\node  (w) at (3,0) {};
  		\end{scope}

  		\begin{scope}
  			\node at (v){$v$};
  			\node at (w) {$w$};
  		\end{scope}

  		\edgeWithTowers{0}{0}{3}{0}{1}{5}

\begin{pgfonlayer}{background}
	\begin{scope}[line width=2pt]
		\draw (v) -- (w);
	\end{scope}
    \begin{scope}[matching]
        \draw (v) -- (1a0);
        \draw (w) -- (1b0);
        \draw (1a1) -- (1b1);
        \draw (1a2) -- (1b2);
        \draw (1a3) -- (1b3);
        \draw (1a4) -- (1a5);
        \draw (1b4) -- (1b5);
    \end{scope}
    \end{pgfonlayer}
\end{tikzpicture} 			}
			\caption{The restriction of the matching $M_2$ to the tower $T$.}
			\label{fig:far-matchings-2}
		\end{subfigure}

		\caption{The extension of the matchings $\tilde{M_1}$ and $\tilde{M_2}$ to a tower of height $5$, as constructed in the proof of \cref{lemma:far-away-matchings}.}
		\label{fig:far-matchings}
	\end{figure}

	Let $\mathcal{C} = (C_1, \dots, C_\ell)$ be a flip sequence from $M_1$ to $M_2$.
	Furthermore assume that none of these cycles is completely contained in the tower $T$, else we are done.
	Hence each $C_i$ is either disjoint from $T$ or it touches $T$.
	Let $c_1 < c_2 < \dots < c_r$ be the indices of the cycles that touch $T$.
	For $i\in \{1,\dots, r\}$ set $N_i = M_1 \Delta C_1 \Delta \dots \Delta C_{c_i}$.
	As $\{a_h,b_h\}\in M_1\setminus M_2$ there is a smallest index $k$, such that $\{a_h, b_h\}\notin N_k$.
	Let $N' = M_1 \Delta C_1 \Delta \dots \Delta C_{c_k-1}$ be the matching before flipping $C_{a_k}$.
	Applying the second statement of \cref{claim:path-properties}, we must have $d(N')= h$, so $\horizontal(N_k) = \emptyset$.
	By the first part of \cref{claim:path-properties} the number of horizontal edges changes by at most one when flipping along a cycle of $\mathcal{C}$.
	Finally observe that the cycles that do not touch $T$ do not change the horizontal edges when we flip them.
	As we start with $h$ horizontal edges in $M_1$, reach $0$ horizontal edges in $N_k$ and then go up to $h-2$ horizontal edges in $M_2$, we have $k\geq h$ and $r-k \geq h-2$, so $r = k + (r-k) \geq 2h-2$, finishing the proof.
\end{proof}

\begin{remark}
	Consider the case $\tilde{M}_1 = \tilde{M}_2$.
	Then the symmetric difference of the matchings $M_1$ and $M_2$ constructed in \cref{lemma:far-away-matchings} is a single alternating cycle.
	So when allowing to flip arbitrary cycles one can easily transform $M_1$ to $M_2$.
	The important difference in the above consideration is, that this cycle can only change the matchings in the given tower.
	On the other hand, a cycle touching $T$ may also touch other towers, allowing us to use a single cycle to modify many towers.
\end{remark}

With the above considerations we are ready to finish the proof of \cref{thm:non-Hamiltonian}.
\bigskip
\par\noindent
\emph{Proof of \cref{thm:non-Hamiltonian}.}\quad
The idea of the proof is to construct two specific matchings in $G_H$ with the following property:
If one can transform one of the matchings to the other by flipping less than $\frac{n}{n-1}(2h-2)$ cycles, then one can use one of these cycles to construct a Hamiltonian cycle in $H$.

We start with the perfect matching $M = \{\{v_1, v_2\}\colon v\in V(H)\}$ in $G_H'$.
We extend $M$ to the towers one after the other by using the construction of \cref{lemma:far-away-matchings}.
For this assume we obtained $G_H$ from $G'_H$ by adding the towers $(T_1,\dots, T_r)$ in this order.
Let $M_1^0=M_2^0=M$.
Assume we constructed two matchings $M_1^i$ and $M_2^i$ in the graph obtained from $G'_H$ by adding the towers $(T_1,\dots, T_i)$.
Now we can apply \cref{lemma:far-away-matchings} to extend these matchings further to $T_{i+1}$ giving rise to two matchings $M_1^{i+1}$ and $M_2^{i+1}$ in the graph obtained from $G'_H$ by adding the towers $(T_1,\dots, T_{i+1})$.
Finally set $M_1=M_1^r$ and $M_2=M_2^r$.

Let us make the following observation that will be useful later:
Consider a tower $T$ and let $G/T$ be the graph we obtain from $G_H$ when replacing the tower $T$ by a single edge.
Then $M_1$ and $M_2$ give rise to two matchings $\tilde{M}_1^T$ and $\tilde{M}_2^T$ in $G/T$, such that applying \cref{lemma:far-away-matchings} to $\tilde{M}_1^T$ and $\tilde{M}_2^T$ gives us $M_1$ and $M_2$, respectively.
So we can use the conclusion of \cref{lemma:far-away-matchings} for every tower $T$ simultaneously, although we constructed the matchings one tower at a time.

Suppose that $M_2$ was reachable from $M_1$ by flipping the cycles $\mathcal{C} = (C_1, \dots, C_m)$ with $m< \frac{n}{n-1}(2h-2)$, i.e.\ $M_2 = M_1 \Delta C_1 \Delta C_2 \Delta \dots \Delta C_m$.
First we consider the set $\mathcal{C}_{s}\subseteq \mathcal{C}$ of cycles, for which there is a tower completely containing this cycle.
We have $|\mathcal{C}_s| \leq |\mathcal{C}| < 4h = t$.
For every $v\in V_H$, we first introduced the edge $e_v = \{v_1,v_2\}$, on which we constructed $t$ towers.
In particular for one of these towers, say $T_v$, no cycle of $\mathcal{C}$ is completely contained in $T_v$.

\begin{figure}[ht]
	\begin{subfigure}[t]{0.5\textwidth}
		\centering
		\scalebox{0.5}{
			\begin{tikzpicture}[scale=1]

	\pgfdeclarelayer{background}
	\pgfdeclarelayer{paths}
	\pgfsetlayers{paths, background,main}
    \newcommand{\displace}{0.5}
	\newcommand{\yoffset}{4}
	\newcommand{\height}{2}

	\pgfmathsetmacro{\wwx}{3}
	\pgfmathsetmacro{\wwy}{-\displace}
	\pgfmathsetmacro{\wx}{-3}
	\pgfmathsetmacro{\wy}{-\displace}

	\pgfmathsetmacro{\uux}{-3 - 0.866*\displace}
	\pgfmathsetmacro{\uuy}{0.5*\displace}
	\pgfmathsetmacro{\ux}{-0.866*\displace}
	\pgfmathsetmacro{\uy}{6*0.866 + 0.5*\displace}

	\pgfmathsetmacro{\vx}{3 + 0.866*\displace}
	\pgfmathsetmacro{\vy}{0.5*\displace}
	\pgfmathsetmacro{\vvx}{0.866*\displace}
	\pgfmathsetmacro{\vvy}{6*0.866 + 0.5*\displace}

	\begin{scope}[every node/.style={thick,draw=black,fill=white,circle,minimum size=17, inner sep=2pt}]

		\node  (v1) at (\vx,\vy) {};
		\node  (v2) at (\vvx,\vvy) {};
		\node  (w2) at (\wwx,\wwy) {};
		\node  (w1) at (\wx,\wy) {};
		\node  (u2) at (\uux,\uuy) {};
		\node  (u1) at (\ux,\uy) {};
	\end{scope}

	\begin{scope}
		\node at (v1){$v_1$};
		\node at (v2){$v_2$};
		\node at (w1){$w_1$};
		\node at (w2){$w_2$};
		\node at (u1){$u_1$};
		\node at (u2) {$u_2$};
	\end{scope}

\begin{pgfonlayer}{background}
	\begin{scope}[line width=2pt]
		\draw (v1) -- (v2);
		\draw (w1) -- (w2);
		\draw (u1) -- (u2);
	\end{scope}

	\begin{scope}[line width=2pt]
		\draw (v1) -- (w2);
		\draw (v1) -- (u2);
		\draw (w1) -- (v2);
		\draw (w1) -- (u2);
		\draw (u1) -- (v2);
		\draw (u1) -- (w2);
	\end{scope}
\end{pgfonlayer}

	\setbool{towerlabels}{false}
    \setcounter{towerIndex}{1}

	\edgeWithTowers{\vvx}{\vvy}{\vx}{\vy}{3}{3}
	\edgeWithTowers{\uux}{\uuy}{\ux}{\uy}{3}{3}
	\edgeWithTowers{\wwx}{\wwy}{\wx}{\wy}{3}{3}

	\begin{pgfonlayer}{background}

	\begin{scope}[matching]
		\draw (v2) -- (1a0);
		\draw (1b0) -- (2a0);
		\draw (2b0) -- (3a0);
		\draw (3b0) -- (v1);

		\draw (w2) -- (7a0);
		\draw (4b0) -- (5a0);
		\draw (5b0) -- (6a0);
		\draw (9b0) -- (w1);

		\draw (u2) -- (4a0);
		\draw (7b0) -- (8a0);
		\draw (8b0) -- (9a0);
		\draw (6b0) -- (u1);

		\draw (1a1) -- (1b1);
		\draw (1a2) -- (1b2);
		\draw (1a3) -- (1b3);

		\draw (2a1) -- (2a2);
		\draw (2b1) -- (2b2);
		\draw (2a3) -- (2b3);

		\draw (3a1) -- (3b1);
		\draw (3a2) -- (3a3);
		\draw (3b2) -- (3b3);

		\draw (4a1) -- (4b1);
		\draw (4a2) -- (4b2);
		\draw (4a3) -- (4b3);

		\draw (5a1) -- (5b1);
		\draw (5a2) -- (5a3);
		\draw (5b2) -- (5b3);

		\draw (6a1) -- (6b1);
		\draw (6a2) -- (6a3);
		\draw (6b2) -- (6b3);

		\draw (7a1) -- (7a2);
		\draw (7b1) -- (7b2);
		\draw (7a3) -- (7b3);

		\draw (8a1) -- (8b1);
		\draw (8a2) -- (8b2);
		\draw (8a3) -- (8b3);

		\draw (9a1) -- (9b1);
		\draw (9a2) -- (9b2);
		\draw (9a3) -- (9b3);

	\end{scope}
\end{pgfonlayer}

\begin{pgfonlayer}{paths}
	\begin{scope}[black!20!white]
		\flippath{v2,1a0,1a1,1b1,1b0,2a0,2a3,2b3,2b0,3a0,3a0,3b0,3b0,v1,
				  w2,7a0,7a3,7b3,7b0,8a0,8a0,8b0,8b0,9a0,9a1,9b1,9b0,w1,
				  u2,4a0,4a1,4b1,4b0,5a0,5a1,5b1,5b0,6a0,6a0,6b0,6b0,u1,
				  v2}{10pt}
	\end{scope}
\end{pgfonlayer}

\end{tikzpicture} 		}
		\caption{}
		\label{subfig:non_hamiltonian_example_1}
	\end{subfigure}
	\begin{subfigure}[t]{0.5\textwidth}
		\centering
		\scalebox{0.5}{
			\begin{tikzpicture}[scale=1]

	\pgfdeclarelayer{background}
	\pgfdeclarelayer{paths}
	\pgfsetlayers{paths, background,main}
    \newcommand{\displace}{0.5}
	\newcommand{\yoffset}{4}
	\newcommand{\height}{2}

	\pgfmathsetmacro{\wwx}{3}
	\pgfmathsetmacro{\wwy}{-\displace}
	\pgfmathsetmacro{\wx}{-3}
	\pgfmathsetmacro{\wy}{-\displace}

	\pgfmathsetmacro{\uux}{-3 - 0.866*\displace}
	\pgfmathsetmacro{\uuy}{0.5*\displace}
	\pgfmathsetmacro{\ux}{-0.866*\displace}
	\pgfmathsetmacro{\uy}{6*0.866 + 0.5*\displace}

	\pgfmathsetmacro{\vx}{3 + 0.866*\displace}
	\pgfmathsetmacro{\vy}{0.5*\displace}
	\pgfmathsetmacro{\vvx}{0.866*\displace}
	\pgfmathsetmacro{\vvy}{6*0.866 + 0.5*\displace}

	\begin{scope}[every node/.style={thick,draw=black,fill=white,circle,minimum size=17, inner sep=2pt}]

		\node  (v1) at (\vx,\vy) {};
		\node  (v2) at (\vvx,\vvy) {};
		\node  (w2) at (\wwx,\wwy) {};
		\node  (w1) at (\wx,\wy) {};
		\node  (u2) at (\uux,\uuy) {};
		\node  (u1) at (\ux,\uy) {};
	\end{scope}

	\begin{scope}
		\node at (v1){$v_1$};
		\node at (v2){$v_2$};
		\node at (w1){$w_1$};
		\node at (w2){$w_2$};
		\node at (u1){$u_1$};
		\node at (u2) {$u_2$};
	\end{scope}

\begin{pgfonlayer}{background}
	\begin{scope}[line width=2pt]
		\draw (v1) -- (v2);
		\draw (w1) -- (w2);
		\draw (u1) -- (u2);
	\end{scope}

	\begin{scope}[line width=2pt]
		\draw (v1) -- (w2);
		\draw (v1) -- (u2);
		\draw (w1) -- (v2);
		\draw (w1) -- (u2);
		\draw (u1) -- (v2);
		\draw (u1) -- (w2);
	\end{scope}
\end{pgfonlayer}

	\setbool{towerlabels}{false}
    \setcounter{towerIndex}{1}

	\edgeWithTowersDifferentHeight{\vvx}{\vvy}{\vx}{\vy}{3}{1,3,0}
	\edgeWithTowersDifferentHeight{\uux}{\uuy}{\ux}{\uy}{3}{1,1,0}
	\edgeWithTowersDifferentHeight{\wwx}{\wwy}{\wx}{\wy}{3}{3,0,1}

 \node at (5a3){};

	\begin{pgfonlayer}{background}

	\begin{scope}[matching]
		\draw (v2) -- (1a0);
		\draw (1b0) -- (2a0);
		\draw (2b0) -- (3a0);
		\draw (3b0) -- (v1);

		\draw (w2) -- (7a0);
		\draw (4b0) -- (5a0);
		\draw (5b0) -- (6a0);
		\draw (9b0) -- (w1);

		\draw (u2) -- (4a0);
		\draw (7b0) -- (8a0);
		\draw (8b0) -- (9a0);
		\draw (6b0) -- (u1);

		\draw (1a1) -- (1b1);

		\draw (2a1) -- (2a2);
		\draw (2b1) -- (2b2);
		\draw (2a3) -- (2b3);

		\draw (4a1) -- (4b1);

		\draw (5a1) -- (5b1);

		\draw (7a1) -- (7a2);
		\draw (7b1) -- (7b2);
		\draw (7a3) -- (7b3);

		\draw (9a1) -- (9b1);

	\end{scope}
\end{pgfonlayer}

\begin{pgfonlayer}{paths}
	\begin{scope}[black!20!white]
		\flippath{v2,1a0,1a1,1b1,1b0,2a0,2a3,2b3,2b0,3a0,3a0,3b0,3b0,v1,
				  w2,7a0,7a3,7b3,7b0,8a0,8a0,8b0,8b0,9a0,9a1,9b1,9b0,w1,
				  u2,4a0,4a1,4b1,4b0,5a0,5a1,5b1,5b0,6a0,6a0,6b0,6b0,u1,
				  v2}{10pt}
	\end{scope}
\end{pgfonlayer}

\end{tikzpicture} 		}
		\caption{}
		\label{subfig:non_hamiltonian_example_2}
	\end{subfigure}
	\begin{subfigure}[t]{.5\textwidth}
		\centering
		\scalebox{0.66}{
			\begin{tikzpicture}[scale=.6]

	\pgfdeclarelayer{background}
	\pgfdeclarelayer{paths}
	\pgfsetlayers{paths, background,main}

	\begin{scope}[every node/.style={thick,draw=black,fill=white,circle,minimum size=17, inner sep=2pt}]
		\node (u) at (-3,0){};
		\node (v) at (3,0){};
		\node (w) at (0, -0.866*6){};
	\end{scope}

	\begin{scope}
		\node at (u){$u$};
		\node at (v){$v$};
		\node at (w){$w$};
	\end{scope}
 \node at (0,2){};

\begin{pgfonlayer}{background}
		\begin{scope}[line width=2pt]
			\draw (u) -- (v);
			\draw (v) -- (w);
			\draw (w) -- (u);
		\end{scope}
		\begin{scope}[line width=2pt]
			\draw (-3,0) to[bend left] (3,0);
			\draw (3,0) to[bend left] (0, -0.866*6);
			\draw (0, -0.866*6) to[bend left] (-3,0);
		\end{scope}
	\end{pgfonlayer}

\begin{pgfonlayer}{paths}
	\draw[line width=13pt, black!20!white] (-3,0) to[bend left] (3,0);
	\draw[line width=13pt, black!20!white] (3,0) to[bend left] (0, -0.866*6);
	\draw[line width=13pt, black!20!white] (0, -0.866*6) to[bend left] (-3,0);
	\end{pgfonlayer}
\end{tikzpicture} 		}
		\caption{}
		\label{subfig:non_hamiltonian_example_3}
	\end{subfigure}
	\begin{subfigure}[t]{0.5\textwidth}
		\centering
		\scalebox{0.66}{
			\begin{tikzpicture}[scale=0.6]

	\pgfdeclarelayer{background}
	\pgfdeclarelayer{paths}
	\pgfsetlayers{paths, background,main}

	\begin{scope}[every node/.style={thick,draw=black,fill=white,circle,minimum size=17, inner sep=2pt}]
		\node (u) at (-3,0){};
		\node (v) at (3,0){};
		\node (w) at (0, -0.866*6){};
	\end{scope}

	\begin{scope}
		\node at (u){$u$};
		\node at (v){$v$};
		\node at (w){$w$};
	\end{scope}
 \node at (0,2){};

	\begin{pgfonlayer}{background}
		\begin{scope}[line width=2pt]
			\draw (u) -- (v);
			\draw (v) -- (w);
			\draw (w) -- (u);
		\end{scope}
	\end{pgfonlayer}

\begin{pgfonlayer}{paths}
	\begin{scope}[black!20!white]
		\flippath{u,v,w,u}{18pt}
	\end{scope}
\end{pgfonlayer}

\end{tikzpicture} 		}
		\caption{}
		\label{subfig:non_hamiltonian_example_4}
	\end{subfigure}
	\caption{Visualization of the idea used in the proof of \cref{thm:non-Hamiltonian}.
		For the sake of simplicity, the number and height of the towers, the matchings and the augmenting paths may differ from the actual choices done in the proof.
		We start with a cycle $C$ in the graph $G_H$ that touches towers on every subdivided edge in $G_H$ (\cref{subfig:non_hamiltonian_example_1}).
		As a first step we remove all vertices that do not lie on $C$ (\cref{subfig:non_hamiltonian_example_2}).
		Next we contract the subpath of $C$ along the subdivision of the edge $\{v_1,v_2\}$ for all $v\in V_H$ (\cref{subfig:non_hamiltonian_example_3}).
		This yields a Hamiltonian cycle $\widetilde{C}$ in $\widetilde{G}$.
		Finally omitting the duplicate edges we end up with a Hamiltonian cycle in $H$ (\cref{subfig:non_hamiltonian_example_4}).}
	\label{fig:visualization_non_hamiltonian}
\end{figure}

For every $v\in V_H$ define $\mathcal{C}_v$ to be the sub-collection of $\mathcal{C}$ consisting of the cycles that touch $T_v$.
By construction of the matchings $M_1$ and $M_2$, using \cref{lemma:far-away-matchings}, we obtain that $|\mathcal{C}_v|\geq 2h-2$ for all $v\in V$.

We claim that every cycle of $\mathcal{C}$ is part of $\mathcal{C}_v$ for at most $n-1$ vertices $v$ of $H$.
Assume the opposite, so there is a cycle $C\in \mathcal{C}$ such that $C\in \mathcal{C}_v$ for all $v\in V(H)$.
We will show, that we can then construct a Hamiltonian cycle in $H$, which contradicts the assumptions of \cref{thm:non-Hamiltonian}.
The cycle $C$ contains edges from $T_v$ for every vertex $v\in V_H$.
In particular, for every vertex $v\in V_H$, $C$ contains an $v_1$-$v_2$ path through the tower gadgets on the edge $\{v_1, v_2\}$.
Now delete all vertices in $V_H \setminus V(C)$.
By the above observation this does not remove any vertex of $V_1$ or $V_2$.
Contracting parts of a cycle still gives a cycle, so if we contract the $v_1$-$v_2$ path along $C$ in both $C$ and $G_H$ for every $v\in V_H$, we still have a cycle.
By the above consideration this contracts a path through the tower gadgets of the edge $\{v_1, v_2\}$.
In particular contracting all these paths in $G_H$ gives rise to a multi-graph $\widetilde{G}$ with $|V_H|$ vertices, namely one for each contraction of the path between $v_1$ and $v_2$ on $C$, for every $v \in V_H$. Furthermore, the contraction $\widetilde{C}$ of the cycle $C$ visits all these vertices.
The only edges we did not contract or delete and thus remain in $\widetilde{G}$ are the edges of the form $\{v_1,w_2\},\{v_2,w_1\}$ for all $\{v,w\}\in E_H$. From this one can see that $\widetilde{G}$ is isomorphic to the graph obtained from $H$ by replacing every edge with a parallel pair of edges. It is easy to see that, as long as both graphs have at least $3$ vertices, any Hamiltonian cycle in $\widetilde{G}$ also induces a Hamiltonian cycle in $H$.
So in particular the contracted cycle $\mathcal{C}$ in $\widetilde{G}$ gives rise to a Hamiltonian cycle in $H$, a contradiction.
A visualization of this idea can be seen in \cref{fig:visualization_non_hamiltonian}.

So we conclude, that every cycle of $\mathcal{C}$ is part of $\mathcal{C}_v$ for at most $n-1$ vertices $v$.
To finish the proof observe that $(n-1)|\mathcal{C}|\geq \sum_{v\in V_H}|\mathcal{C}_v| \geq n(2h-2)$, so $|\mathcal{C}|\geq \frac{n}{n-1}(2h-2)$.

This shows that $M_1$ and $M_2$ correspond to vertices of the perfect matching polytope of $G_H$ at distance at least $\frac{n}{n-1}(2h-2)$, and thus $\mathrm{diam}(P_{G_H})\ge \frac{n}{n-1}(2h-2)$, as claimed.
\qed

\section{Proof of \cref{thm:monotoneperfectmatchingpolytope} via \cref{lem:4dm-reduction}}\label{sec:proof-monotone-diameter}

The remainder of this article is dedicated to the proof our second main result, \cref{thm:monotoneperfectmatchingpolytope}, based on the reduction described in \cref{sec:reduction-monotone-diameter}.
We start by proving the description of the monotone diameter of the perfect matching polytope given in \cref{lemma:monotone-to-symmetric-difference}.

\bigskip
\par\noindent
\emph{Proof of \cref{lemma:monotone-to-symmetric-difference}.}\quad
First let ${x}$ be a vertex of $P_G$ corresponding to the perfect matching $M$, and let ${c}\in \mathbb{R}^E$ be a cost function.
Let $M^*$ be a minimum cost perfect matching in $G$ with respect to the costs $c$, corresponding to the $c$-optimal vertex $y$.
Consider the cycles $\mathcal{C}$ of the symmetric difference $M\Delta M^*$.
Flipping the edges of any cycle of $\mathcal{C}$ with in $M$ does not increase the costs with respect to $c$, as $M^*$ is a minimum cost matching.
So in particular flipping the cycles of $\mathcal{C}$ that reduce the cost one after another gives rise to a monotone walk of length $|\mathcal{C}|$ from $x$ to an optimal vertex $y'$.
Note that due to cycles with an alternating cost of $0$ the optimal vertex we reach might differ from the vertex we used to guide the path.
This is not an issue, as the definition of monotone diameter only uses the distance to any optimal vertex.

As $\mathbf{x}$ and $\mathbf{c}$ were arbitrary, this proves that the monotone diameter of $P_G$ is bounded from above by the maximum number of cycles in the symmetric difference of two perfect matchings.

In order to prove the matching lower bound it is enough to show the following:
Let $M$ and $M^*$ be two perfect matchings, corresponding to the vertices $\chi^M$ and $\chi^{M^*}$ of $G$, respectively,  and let $\mathcal{C}$ denote the cycles of the symmetric difference $M\Delta M^*$.
Then there exists a cost function $c\in \mathbb{R}^E$ such that $\chi^{M^*}$ is $c$-minimal and such that the shortest $c$-monotone walk from $\chi^M$ to $\chi^{M^*}$ has length $|\mathcal{C}|$.

We will show that defining $c$ as
\begin{equation*}
	c(e) =
	\begin{cases}
		0 & \text{ if } e\in M^*,\\
		1 & \text{ if } e\in M\setminus M^*,\\
		|V| & \text{ if } e\in E\setminus\{M\cup M^*\},
	\end{cases}
\end{equation*}
fulfills the above requirements.
First, as $c$ is non-negative and $c(M^*)=0$, $M^*$ is indeed $c$-minimal.
Next observe that $c(M) \leq |M| < |V|$.
So if we consider an arbitrary monotone walk from $\chi^M$ to $\chi^{M^*}$, corresponding to flipping a sequence of negative alternating cycles, we may not use an edge of $E\setminus\{M\cup M^*\}$.
If we did, then one of the intermediate vertices would correspond to a matching containing an edge of this set.
Then the cost of that vertex is at least $|V|$, contradicting the monotonicity of the walk.

Hence the cycles corresponding to the moves of any monotone walk from $\chi^M$ to $\chi^{M^*}$ may only use the edges of $M\cup M^*$.
As the edges of $M\cap M^*$ are isolated in $M\cup M^*$, they cannot be part of a cycle in $M\cup M^*$.
Hence, we can strengthen the observation and the cycles corresponding to the moves may even only use the edges of $M\Delta M^*$.
In particular, every monotone walk from $\chi^M$ to $\chi^{M^*}$ has to flip the cycles of $M\Delta M^*$ one after another.
Thus, every monotone walk has length $|\mathcal{C}|$, finishing the proof.
\qed

Recall, that, based on \cref{lemma:monotone-to-symmetric-difference}, we showed, that in order to prove \cref{thm:monotoneperfectmatchingpolytope} it suffices to prove \cref{lem:4dm-reduction}.
We already proposed a reduction there which we will recall and make more precise below.

Let $W,X,Y,Z$ and $E$ be an instance of \textsc{4-Dimensional Matching}.
We will construct an instance of \textsc{Vertex-Disjoint $4$-Cycle Cover}, such that the answers to both problems agree. To do so we construct a bipartite graph $G$ in the following way:
For every element $a\in W\cup X\cup Y\cup Z$ we add a vertex to $G$. These vertices will be called the \emph{exterior vertices} of $G$.
Furthermore, for every element $e \in E$ we add twelve vertices and 28 edges forming a gadget as shown in \cref{fig:4dm-reduction-gadget-recall}.
More precisely, for every hyperedge $e=(w,x,y,z)\in E$ we add vertices $a_i^e$ for $a\in e$ and $i\in \{1,2,3\}$, which we will call the \emph{auxiliary vertices} of $G$ in the following.
Additionally we add the edges of
\[
E_{ext} = \{\{a, a_1^e\}, \{a^e_1, a^e_2\}, \{a^e_2,a^e_3\}, \{a^e_3, a\}\colon a\in \{w,x,y,z\}\}
\]
and
\[
E_{int} = \{\{w_i^e, x_i^e\}, \{x_i^e, y_i^e\}, \{y_i^e, z_i^e\}, \{z_i^e, w_i^e\}\colon i\in \{1,2,3\}\}.
\]
Observe that the edges of $E_{ext}$ form four vertex-disjoint cycles of length four, marked in black and solid in \cref{fig:4dm-reduction-gadget}.
The edges of $E_{int}$ form three vertex-disjoint cycles of length four, depicted in red and dashed in \cref{fig:4dm-reduction-gadget}.
Furthermore both the first and the second set of cycles cover all auxiliary vertices.
The first set additionally covers the external vertices.

The above construction indeed gives a bipartite graph, using the following bipartition:
\begin{align*}
	\begin{array}{rllll}
		U =&& \{x, x_2^e \colon x\in X, e\in E, x\in e\} &\cup& \{y_1^e, y_3^e \colon y\in Y, e\in E, y\in e\} \\ &\cup& \{z, z_2^e \colon z\in Z, e\in E, z\in e\} &\cup& \{w_1^e, w_3^e \colon w\in W, e\in E, w\in e\}
	\end{array}\\
	\begin{array}{rllll}
		V =&& \{x_1^e, x_3^e \colon x\in X, e\in E, x\in e\} &\cup& \{y, y_2^e \colon y\in Y, e\in E, y\in e\} \\ &\cup& \{z_1^e, z_3^e \colon z\in Z, e\in E, z\in e\} &\cup& \{w, w_2^e \colon w\in W, e\in E, w\in e\}.
	\end{array}
\end{align*}

\begin{figure}
	\centering
	\begin{tikzpicture}[scale=1.2]
        \newcommand{\offset}{0.5}
        \begin{scope}[every node/.style={thick,draw=black,fill=white,circle,minimum size=15, inner sep=1pt}]
                \node[fill=black!20!white] (w) at (-\offset ,-\offset) {$w$};
                \node (x) at (-\offset ,4+\offset) {$x$};
                \node[fill=black!20!white] (y) at (4+\offset ,4+\offset) {$y$};
                \node (z) at (4+\offset ,-\offset) {$z$};
        \end{scope}
        
        \begin{scope}[every node/.style={thick,draw=black,fill=white,circle,minimum size=5, inner sep=1pt}]
                \node (w1) at (1 ,0) {};
                \node[fill=black!20!white] (w2) at (1 ,1) {};
                \node (w3) at (0 ,1) {};
                
                \node[fill=black!20!white] (x1) at (0 ,3) {};
                \node (x2) at (1 ,3) {};
                \node[fill=black!20!white] (x3) at (1 ,4) {};
                
                \node (y1) at (3 ,4) {};
                \node[fill=black!20!white] (y2) at (3 ,3) {};
                \node (y3) at (4 ,3) {};
                
                \node[fill=black!20!white] (z1) at (4 ,1) {};
                \node (z2) at (3 ,1) {};
                \node[fill=black!20!white] (z3) at (3 ,0) {};
        \end{scope}
        
        \begin{scope}[every node/.style={thick,draw=black,fill=white,circle,minimum size=5, inner sep=1pt, font=\scriptsize}]
                \node (w1) at (1 ,0) {$w^e_1$};
                \node[fill=black!20!white] (w2) at (1 ,1) {$w^e_2$};
                \node (w3) at (0 ,1) {$w^e_3$};
                
                \node[fill=black!20!white] (x1) at (0 ,3) {$x^e_1$};
                \node (x2) at (1 ,3) {$x^e_2$};
                \node[fill=black!20!white] (x3) at (1 ,4) {$x^e_3$};
                
                \node (y1) at (3 ,4) {$y^e_1$};
                \node[fill=black!20!white] (y2) at (3 ,3) {$y^e_2$};
                \node (y3) at (4 ,3) {$y^e_3$};
                
                \node[fill=black!20!white] (z1) at (4 ,1) {$z^e_1$};
                \node (z2) at (3 ,1) {$z^e_2$};
                \node[fill=black!20!white] (z3) at (3 ,0) {$z^e_3$};
        \end{scope}

        \begin{scope}[line width=3pt]
            \draw (w) -- (w1);
            \draw (w1) -- (w2);
            \draw (w2) -- (w3);
            \draw (w3) -- (w);

            \draw (x) -- (x1);
            \draw (x1) -- (x2);
            \draw (x2) -- (x3);
            \draw (x3) -- (x);

            \draw (y) -- (y1);
            \draw (y1) -- (y2);
            \draw (y2) -- (y3);
            \draw (y3) -- (y);

            \draw (z) -- (z1);
            \draw (z1) -- (z2);
            \draw (z2) -- (z3);
            \draw (z3) -- (z);
        \end{scope}

        \begin{scope}[line width=3pt, orange!80!red, dashed, dash pattern=on 5pt off 2pt]
            \draw (w1) -- (x1) -- (y1) -- (z1) -- (w1);
            \draw (w2) -- (x2) -- (y2) -- (z2) -- (w2);
            \draw (w3) -- (x3) -- (y3) -- (z3) -- (w3);
        \end{scope}

    \end{tikzpicture} 	\caption{Reminder of the gadget for a hyperedge $e=(x,y,z,w)$, used in the proof of \cref{lem:4dm-reduction}.
		The nodes marked in gray form one side of a bipartition.}
	\label{fig:4dm-reduction-gadget-recall}
\end{figure}

\par\noindent
\emph{Proof of \cref{lem:4d-to-cycle-cover}.}\quad
It remains to show that the original instance $(W\cup X \cup Y \cup Z,E)$ has a $4$-dimensional matching if and only if $G$ contains a collection of $4$-cycles that cover every vertex precisely once.

First assume, that there is a $4$-dimensional matching $M$ in $(W\cup X \cup Y \cup Z,E)$.

To construct the collection of $4$-cycles, we will consider every gadget individually.
So let $G_e$ be the subgraph of $G$ corresponding to the gadget of an edge $e=(w,x,y,z)\in E$.
If $e\in M$ we take the cycles defining $E_{ext}$, four in total, else we take the cycles of $E_{int}$, three in total.
In both cases all twelve auxiliary vertices of the gadget are covered.
Additionally, as $M$ is a 4-dimensional matching, every element of $W,X, Y$ and $Z$ is part of exactly one edge and hence the corresponding vertex is covered exactly once.
Thus we indeed constructed a set of $4$-cycles that covers every vertex exactly once.

For the reverse direction assume that there is a set $\mathcal{C}$ of 4-cycles in $G$ such that every vertex is covered exactly once.
First note that the distance between any two distinct exterior vertices in $G$ is at least $3$, and thus no cycle of length $4$ in $G$ can cover more than one exterior vertex.
In particular no cycle of length $4$ can use vertices from different gadgets.
Next observe, that a gadget contains $16$ vertices, counting in the exterior vertices.
If we consider the cycles of $\mathcal{C}$ that completely lie within the gadget of $e$, they cover a number of vertices that is divisible by $4$.
By the above they have to cover all twelve auxiliary vertices.
In particular either twelve or sixteen vertices are covered, i.e.\ either all exterior vertices are covered by cycles within the gadget of $e$, or no exterior vertex is.

Consider the set $M\subseteq E$ consisting of all edges, for which the cycles contained in the corresponding gadget cover all exterior vertices.
We claim that $M$ covers every element of $W \cup X \cup Y \cup Z$ exactly once, i.e., forms a 4-dimensional matching.
To see this, consider any fixed exterior vertex $v \in W \cup X \cup Y \cup Z$.
Then $v$ is covered by a cycle, which lies completely in the gadget of some edge $e$.
As discussed above this implies that within the gadget of $e$ all four exterior vertices are covered.
Hence we have $e\in M$.
On the other hand, for every other $e\in E$ that contains the element corresponding to $v$, the gadget of $e$ does not contain a cycle covering $v$, so $e\notin M$.
So every element of $W, X, Y$ and $Z$ is covered exactly once by $M$, finishing the reduction.
\qed

\paragraph*{Acknowledgement.} The second author would like to thank Jean Cardinal for bringing the open problem of the complexity of circuit diameters to his attention, and for lively discussions on this topic at ETH Zurich in 2022, that sparked the work presented in this paper. We also would like to thank Steffen Borgwardt and Weston Grewe for comments on the manuscript.

\printbibliography

\end{document}